\documentclass{amsart}
\usepackage{amsmath}
\usepackage{amsthm}
\usepackage{amssymb}
\usepackage{color}
\usepackage{xcolor}

\newtheorem{thm}{Theorem}
\newtheorem{lem}[thm]{Lemma}

\newtheorem{defi}[thm]{Definition}
\newtheorem{prop}[thm]{Proposition}
\newtheorem{rk}[thm]{Remark}

\newcommand{\rr}{{\mathbb{R}}}
\newcommand{\rd}{{\rr^3}}
\newcommand{\dd}{\mathrm {d}}
\newcommand{\e}{\varepsilon}
\newcommand{\indiq}{\hbox{\rm 1}{\hskip -2.8 pt}\hbox{\rm I}}
\newcommand{\Id}{\mathbf{I}_3}
\newcommand{\E}{\mathbb{E}}
\newcommand{\PR}{\mathbb{P}}
\newcommand{\diver}{\mathrm {div}}
\newcommand{\Tr}{\mathrm{Tr}}

\newcommand{\intot}{\int_0^t }
\newcommand{\intrd}{\int_{\rr^3}}
\newcommand{\intrdrd}{\int_{\rr^3\times\rr^3}}
\newcommand{\cP}{{\mathcal P}}

\newcommand{\cL}{{\mathcal L}}

\newcommand{\cW}{{\mathcal W}}
\newcommand{\cLL}{{\mathcal A}}
\newcommand{\cH}{{\mathcal H}}
\newcommand{\cT}{{\mathcal T}}
\newcommand{\Law}{{\rm Law}}
\newcommand{\tV}{\tilde V}

\newcommand{\tf}{\tilde f}
\newcommand{\tv}{\tilde v}
\newcommand{\tk}{\tilde k}
\newcommand{\tx}{\tilde x}
\newcommand{\xs}{x_*}
\newcommand{\vs}{v_*}
\newcommand{\tvs}{\tilde v_*}
\newcommand{\lps}{\langle \!\langle}
\newcommand{\rps}{\rangle\!\rangle}
\newcommand{\vip}{\vskip.13cm}

\newcommand{\bla}{\color{black}}

\begin{document}
\title[Landau equation for hard potentials]{ Stability, well-posedness and regularity of 
the homogeneous Landau equation for hard potentials }
%\title{Stability of the homogeneous Landau equation 
%for hard potentials}

\author{Nicolas Fournier}
\author{Daniel Heydecker}

\address{N. Fournier : Sorbonne Universit\'e, LPSM-UMR 8001, Case courrier 158,75252 Paris Cedex 05, France.}
\email{nicolas.fournier@sorbonne-universite.fr}
\address{D. Heydecker : University of Cambridge, Centre for Mathematical Sciences, Wilberforce Road, CB30WA, United Kingdom.}
\email{dh489@cam.ac.uk.}



\begin{abstract}
 We establish the well-posedness and some quantitative stability of the spatially homogeneous 
Landau equation 
for hard potentials, using some specific Monge-Kantorovich cost, 
assuming only that the
initial condition is a probability measure with a finite moment of order
$p$ for some $p>2$.
As a consequence, we extend previous regularity results and 
show that all non-degenerate measure-valued solutions to the Landau equation,
with a finite initial energy,
immediately admit analytic densities with finite entropy.
Along the way, we prove that the Landau equation instantaneously creates Gaussian moments.
We also show existence of weak solutions 
under the only assumption of finite initial energy. 
\end{abstract}

\maketitle

\section{Introduction and main results}

\subsection{The Landau equation} We study the spatially homogeneous (Fokker-Planck-)Landau equation, which governs the time-evolution of the distribution $f_t, t\ge 0$ of velocities in a plasma: \begin{align} \label{LE}
\partial_t f_t(v) = \frac 1 2 
\diver_v \Big( \intrd a(v-\vs)[ f_t(\vs) \nabla f_t(v) - f_t(v) \nabla f_t(\vs)
]\,\dd \vs \Big)
\end{align} where $a$ is the nonnegative, symmetric matrix \[
a(x) = |x|^{2+\gamma} \Pi_{x^\perp}; \quad \Pi_{x^\perp}=  \Id - \frac{xx^*}{|x|^2}
\] and $\gamma \in [-3,1]$ parametrises a range of models, depending on the interactions between particles. While the most physically relevant case is $\gamma=-3$, which models Coulomb interaction, we will study the cases $\gamma \in (0,1]$ of \emph{hard potentials}, where the Landau equation 
\eqref{LE} may be understood as a limit of the Boltzmann equation in the asymptotic of grazing collisions,
see Desvillettes \cite{d2} and Villani \cite{v:nc,v:h}. 
\vip 
This equation was studied in detail by Desvillettes and Villani \cite{dv1,dv2}, who give results on existence, uniqueness, regularising effects and large-time behavior. Regarding stability, we refer to \cite{fgui}, on which the present work builds. Let us also mention the work of Carrapatoso \cite{c} on exponential convergence to equilibrium, some recent works of Chen, Li and Wu \cite{ch,ch2} and Morimoto, Pravda-Starov and Xu \cite{morimoto}
extending the regularity results, as well as the recent gradient flow approach by 
Carrillo, Delgadino, Desvillettes and Wu \cite{cddw}. \bla

\subsection{Notation}
We denote by $\cP(\rr^3)$ the set of probability measures on $\rd$, and for $p>0$, we set $\cP_p$ to be those probability measures with a finite $p^\text{th}$ moment: $\cP_p(\rd)=\{f\in\cP(\rd)\;:\;m_p(f)<\infty\}$, where $m_p(f)=\intrd |v|^p f(\dd v)<\infty$.

\vip

We will use the following family of transportation costs to measure the distance between two solutions.
For $p>0$ and $f,\tf\in\cP_p(\rd)$, we write $\cH(f,\tf)$ for the set of all couplings  
$$\cH(f,\tf) = \bigl\{    R \in \cP(\rr^3\times\rr^3) \; : \;
R  \text{ has marginals } f \text{ and } \tf \bigr\}.$$ 
 With this notation, we define the optimal transportation cost$$
\cT_{p}(f,\tf)= \inf \Big\{\intrdrd (1+|v|^p+|\tv|^p)\frac{|v-\tv|^2}{1+|v-\tv|^2}
R(\dd v,\dd \tv) \; : \; R \in \cH(f,\tf)   \Big\}.
$$

The form of this optimal transportation cost is key to our stability and uniqueness arguments; the major improvement in Theorem \ref{main} below relies on a negative term which appears due to a {\it Pozvner effect} 
of the prefactor $(1+|v|^p+|\tv|^p)$. Note that for each $p\geq 2$, there is a constant $C>0$ such that 
$|v-\tv|^p \leq C(1+|v|^p+|\tv|^p)\frac{|v-\tv|^2}{1+|v-\tv|^2}$, so that for $\cW_p$ the usual
Wasserstein distance of order $p$, we have 
$\cW_p^p \leq C \cT_p$. It can also be checked that convergence in $\cW_p$ implies convergence in $\cT_p$. It follows that both $\cW_p$ and $\cT_p$ \bla
generate the same topology, equivalent to weak convergence plus convergence of 
the $p^\text{th}$ moments.

\vip

We will also consider regularity of solutions. For $k, s\ge 0$, we define the weighted Sobolev norm 
$$ \|u\|_{H^k_s(\rd)}^2=\sum_{|\alpha|\le k}\intrd |\partial_\alpha u(v)|^2(1+|v|^2)^{s/2}  \dd v
$$ 
and the weighted Sobolev space $H^k_s(\rd)$ for those $u$ where this is finite. By an abuse of notation, we say that $f\in \cP(\rd)$ belongs to $H^k_s(\rd)$ if $f$ admits a density $u$ with respect to the Lebesgue measure with $u\in H^k_s(\rd)$, and in this case we write $\|f\|_{H^k_s(\rd)}=\|u\|_{H^k_s(\rd)}$. Similarly, we say that $f\in \cP(\rd)$ is analytic if $f$ admits an analytic density.

\vip 

We finally define the entropy $H(f)$ of a probability measure $f\in \cP(\rd)$ by 
$$ H(f)=\begin{cases} \intrd u(v)\log u(v) \dd v & \text{if }f \text{ has a density }u; \\ \infty & \text{otherwise.} \end{cases} $$

\subsection{Weak solutions}

We define, for $x\in \rd$,
\begin{equation}\label{db}
b(x)=\diver \; a(x)=-2 |x|^\gamma x.
\end{equation}
For $(f_t)_{t\geq 0}$ a family of probability measures on $\rd$ and for $p,q>0$, we say that
$(f_t)_{t\geq 0}$ belongs to $L^\infty_{loc}([0,\infty),\cP_{p}(\rd)) \cap L^1_{loc}([0,\infty),\cP_{q}(\rd))$ if
$$
\sup_{t \in [0,T]} m_p(f_t)+\int_0^T m_q(f_t) \dd t <\infty \quad \hbox{for all $T>0$.}
$$
We will use the following classical notion of weak solutions,
see Villani \cite{v:nc} and Goudon \cite{gou}.

\begin{defi}\label{ws}
Let $\gamma \in (0,1]$. We say that $(f_t)_{t\geq 0}$ is a weak solution to \eqref{LE} if it 
belongs to
$L^\infty_{loc}([0,\infty),\cP_{2}(\rd)) \cap L^1_{loc}([0,\infty),\cP_{2+\gamma}(\rd))$, if
$m_2(f_t)\leq m_2(f_0)$ for all $t\geq 0$ and if 
for all $\varphi\in C^2_b(\rr^3)$, all $t\geq 0$,
\begin{align}\label{wf}
\intrd \varphi(v)f_t(\dd v) = \intrd \varphi(v)f_0(\dd v) + \intot \intrd \intrd \cL\varphi(v,\vs) 
f_s(\dd \vs)f_s(\dd v) \dd s,
\end{align}
where
\begin{equation} \label{L}
\cL\varphi(v,\vs)= \frac 1 2 \sum_{k,\ell=1}^3 a_{k\ell}(v-\vs)\partial^2_{k\ell}\varphi(v)+ \sum_{k=1}^3
b_{k}(v-\vs)\partial_{k}\varphi(v).
\end{equation}
\end{defi}

Since $|\cL\varphi(v,\vs)|\leq C_\varphi (1+|v|+|\vs|)^{2+\gamma}$ for $\varphi\in C^2_b(\rr^3)$,
every term makes sense in \eqref{wf}. 

\subsection{Existence and properties of weak solutions}
First we summarise some results of Desvillettes and Villani.

\begin{thm}[Desvillettes \& Villani, Theorems 3 and 6 in \cite{dv1}]\label{ddv}
Fix $\gamma \in (0,1]$, $p\geq 2$ and $f_0\in \cP_p(\rd)$. \vip
(a) If $(f_t)_{t\ge 0}$ is any weak solution to \eqref{LE} starting at $f_0$, then we have conservation of the kinetic energy, i.e. $m_2(f_t)= m_2(f_0)$ for all $t\geq0$,
and the estimates $\sup_{s\in[0,\infty)} m_p(f_s)<\infty$ and
$\int_0^t m_{p+\gamma}(f_s)\dd s <\infty$ for all $t\geq 0$. Further, for all $q>0$ and 
$t_0>0$, $\sup_{t\geq t_0}m_q(f_t)<\infty$.
\vip
(b)  If $p>2$, then a weak solution starting at $f_0$ exists.
\vip 
(c)  If $p>2$ and if $f_0$ is not concentrated on a line, then there exists a weak solution 
$(f_t)_{t\ge 0}$ starting at $f_0$ such that for all $t>0$, $f_t$ has finite entropy 
$H(f_t)<\infty$ and 
\begin{equation} \label{eq: sobolev regularity} 
\hbox{for all $k, s\ge 0$ and all $t_0>0$,} 
\quad \sup_{t\ge t_0} \|f_t\|_{H^k_s(\rd)}<\infty.
\end{equation}
\end{thm}

Let us remark that the cited theorem makes the additional assumption in (a) that 
$f_t$ has a density for all
$t\geq 0$, but this is not used in the proof. Regarding (b), the cited theorem assumes 
that $f_0$ has a density,
but this is only required to show that the weak solution they build has
a density, for all times. Concerning (c), Desvillettes and Villani also assume that 
$f_0$ has a density, but only use that $f_0$ is not concentrated on a line,
see the remark under Lemma 9 of the cited work. To be more explicit, \emph{$f_0$ not concentrated on a line} means 
that for any $x_0,u_0 \in \rd$, 
setting $L=\{x_0+\lambda u_0 : \lambda \in \rr\}$, there holds $f_0(\rd \setminus L)>0$.

\vip 
Regarding existence, we are able to prove the following extension to (b) above, 
removing the condition that $f_0 \in \cP_{p}(\rd)$ for some $p>2$ 
and requiring only $f_0\in \cP_2(\rd)$. 

\begin{thm}\label{mainexist} Let $\gamma \in (0,1]$ and \bla $f_0\in \cP_2(\rd)$. 
There exists a weak solution to \eqref{LE} starting at $f_0$. 
\end{thm} 

Let us now state the following strengthening of (c), due to Chen, Li and Xu. 

\begin{thm}[Chen, Li \& Xu, Theorem 1.1 in \cite{ch2}]\label{analytic regularity} 
 Fix $\gamma \in (0,1]$.  Let $(f_t)_{t\ge 0}$ be a weak solution to \eqref{LE} such that the estimate \eqref{eq: sobolev regularity} holds. Then $f_t$ is analytic for all $t>0$. \end{thm}  

Our main result on regularity is as follows, and shows that the conclusions above apply to all weak solutions to \eqref{LE}, aside from the degenerate case 
of point masses.  

\begin{thm}\label{mainregularity} Fix $\gamma \in (0,1]$. \bla 
Let $f_0\in \cP_2(\rd)$ be a probability measure which is not a 
Dirac mass, and let $(f_t)_{t\ge 0}$ be any weak \bla solution to \eqref{LE} starting at $f_0$. 
 Then 
the estimate \eqref{eq: sobolev regularity} holds and for all 
$t>0$, $f_t$ is analytic and has a finite entropy.
\end{thm} 

We emphasise that Theorem \ref{mainregularity} applies to \emph{any} weak solution, 
while Theorems \ref{ddv}-(c) and 
\ref{analytic regularity} only show that there exists such a regular solution
(see the remarks after Theorem 6 in \cite{dv1}). 
This follows from Theorem \ref{main} below, although we are not able to prove uniqueness
under the sole assumption that $f_0 \in \cP_2(\rd)$.
Let us also remark that, in the excluded case where $f_0=\delta_{v_0}$ is a point mass, 
then the unique solution is  $f_t=\delta_{v_0}$ for all $t\ge 0$ by conservation of energy 
and momentum, and so there is no hope of regularity.

\vip 

As a step towards our main stability result below, we will prove the following proposition, 
which improves on the appearance of moments in item (a) above and may be of independent interest.
\begin{prop}\label{expo}
Fix $\gamma \in (0,1]$ and consider a weak solution $(f_t)_{t\geq 0}$ to \eqref{LE}.
There are some constants $a>0$ and $C>0$, both depending only on $\gamma$ and $m_2(f_0)$, such that
$$
\intrd e^{a|v|^2}f_t(\dd v) \leq C\exp[Ct^{-2/\gamma}] \quad \hbox{for all $t>0$.}
$$
\end{prop}

Since the preliminary version of this work, the first author has studied the 
Boltzmann equation with hard potentials and without cutoff, which produces some 
exponential moments of the form  $\intrd e^{a|v|^\rho}f_t(dv)$, with 
$\rho \in (\gamma,2]$ depending on the singularity of the angular 
collision kernel. In the case with cutoff, only exponential moments of 
the form $\intrd e^{a|v|^\gamma}f_t(dv)$ become finite for $t>0$, see Alonso-Gamba-Taskovic
\cite{agt}.

\subsection{Uniqueness and stability} Let us mention the following result, due to the 
first author and 
Guillin, which can be compared to our result and on which we build.

\begin{thm}[Fournier \& Guillin, Theorem 2 in \cite{fgui}]\label{uniqueness}  
Fix $\gamma \in (0,1]$ and let $f_0\in \cP(\rd)$ be such that 
\begin{equation}\label{emfg}
\mathcal{E}_\alpha(f_0)=\intrd e^{|v|^\alpha}f_0(\dd v)<\infty \quad \hbox{for some $\alpha >\gamma$.} 
\end{equation}
Then there exists a unique weak solution $(f_t)_{t\ge 0}$ to \eqref{LE} starting at $f_0$. 
Moreover, 
if $\eta \in (0,1)$, $ \lambda\in (0,\infty)$ and $T>0$,
then there exists a constant $C=C(T,\eta,\mathcal{E}_\alpha(f_0),\lambda)$ such that, 
if $(\tf_t)_{t\ge 0}$ is another solution satisfying $\sup_{t\in[0,T]} m_{2+\gamma}(\tf_t)
\le \lambda$, 
then $$ \sup_{t\in[0,T]}\mathcal{W}_2(f_t, \tf_t) \le C [\mathcal{W}_2(f_0, \tf_0)]^{1-\eta} $$ 
where $\mathcal{W}_2$ is the usual Wasserstein distance with quadratic cost.
\end{thm}

The main result of this paper is the following, which consists in relaxing the condition \eqref{emfg}
 and in replacing, {\it via} another transportation cost, 
the H\"older dependance in the initial condition
by some Lipschitz dependance.

\begin{thm}\label{main}
Fix $\gamma \in (0,1]$ and $p>2$ and two weak solutions $(f_t)_{t\geq 0}$ and
$(\tf_t)_{t\geq 0}$ to \eqref{LE} starting from $f_0$ and $\tf_0$, both belonging to
$\cP_{p}(\rd)$.
There is a constant $C$, depending only on $p$ and $\gamma$, such that for all $t\geq 0$,
\begin{equation}
\cT_p(f_t,\tf_t)\leq \cT_p(f_0,\tf_0) \exp\Big(C \Big[1+\sup_{s\in[0,t]}m_p(f_s+\tf_s)\Big]
\Big[1+\int_0^t (1+m_{p+\gamma}(f_s+\tf_s))\dd s\Big] \Big).
\label{eq: conclusion of main} \end{equation}
\end{thm}

Together with Theorem \ref{ddv}, this shows that when $f_0 \in \cP_p(\rd)$ for some $p>2$,
\eqref{LE} has a unique weak solution
and this provides a quantitative stability estimate.

\subsection{Discussion}
This current paper is primarily concerned with stability, 
continuing the \bla previous analyses of the Cauchy problem for the Landau equation with hard potentials 
by Arsen’ev-Buryak \cite{ar}, Desvillettes-Villani \cite{dv1}, see also \cite{fgui}. 
In addition to the mathematical interest of uniqueness and stability, these are physically relevant criteria: if the equation is not well-posed, then it cannot be a complete description of the system and additional information is needed. Let us also note that stability estimates play a key role in the functional framework of 
Mischler-Mouhot \cite{mm} and Mischler-Mouhot-Wennberg \cite{mmw}  for proving \emph{propagation of chaos} for interacting particle systems, and these have been applied  to Kac's process \cite{k}. See 
also the work of Norris \cite{n}
and \cite{h1}.  
In this context, it is particularly advantageous that our result requires neither regularity nor exponential moments, as these are not readily applicable to the empirical measures of the particle system.  It is also satisfying to get a Lipschitz dependance in the initial condition, so that error terms
will not increase too much as time evolves.

\vip 

The study of stability via coupling, on which this work builds, goes back to Tanaka \cite{t} for the Boltzmann equation in the case of Maxwell molecules; let us mention the later works \cite{fm,fmi,h2} which apply the same principle in the context of hard potentials.  The same idea was applied to the Landau equation by Funaki \cite{f} and has previously been applied by the first author \cite{fgui} in the context of stability and propagation of chaos. 
 See \cite{fp} for a review of \bla coupling methods for PDEs. 

\vip

Compared to the previous literature regarding  uniqueness and stability  for the Landau equation with hard potentials, our main result is substantially stronger and more general. The  uniqueness result  
of Desvillettes and Villani \cite{dv1} requires
that the initial data $f_0$ has a density $u_0$ satisfying 
\begin{equation}\label{cudv}  
\intrd (1+|v|^2)^{p/2}u_0^2(v) \dd \bla v<\infty \quad \hbox{for some $p>15+5\gamma$,}
\end{equation}
while the result of
the \cite{fgui} recalled in Theorem \ref{uniqueness} above allows measure solutions, but requires a 
finite exponential moment. Our result therefore allows much less localisation than either of the results above, 
while also not requiring any regularity on the initial data $f_0, \tf_0$.

\vip

The case $\gamma=0$ of Maxwell molecules is particularly simple, and results of Villani \cite{v:max} show existence and uniqueness only assuming finite energy, i.e. that $f_0\in\cP_2(\rd)$.
Note that the Boltzmann equation for hard potentials with cutoff has also 
been shown to be well-posed by Mischler-Wennberg \cite{mw} as 
soon as $f_0\in\cP_2(\rd)$, by a completely different method breaking down in the case 
without cutoff.
Hence our condition on $f_0$ for Theorem \ref{main}, namely the finiteness of a moment of order $p>2$, seems almost optimal. While this is feasable for existence,
we did not manage to prove uniqueness assuming only a finite initial energy.

\vip

To summarize, our uniqueness and stability statement Theorem \ref{main} is
much stronger than the previous results and almost optimal, since we assume
that $f_0 \in \cP_{2+}(\rd)$ instead of \eqref{cudv} as in \cite{dv1} or
\eqref{emfg} as in \cite{fgui};
we slightly improve in Theorem \ref{mainexist}
the existence result of \cite{dv1}, assuming that $f_0 \in \cP_2(\rd)$
instead of $f_0 \in \cP_{2+}(\rd)$; we are able to prove in Theorem
\ref{mainregularity} the smoothness
of \emph{any} weak solution with $f_0 \in \cP_2(\rd)$,
instead of showing the existence of one smooth solution when $f_0 \in \cP_{2+}(\rd)$
as in \cite{dv1} and \cite{ch2}; and we prove the appearance of some Gaussian moments
for any weak solution with $f_0 \in \cP_2(\rd)$.
\bla

\vip

Let us finally mention that in the case $\gamma=0$, a stronger `ulta-analytic' regularity is 
known, see Morimoto, Pravda-Starov and Xu \cite{morimoto}; 
in this case, one has the advantage that 
the coefficients of \eqref{LE} are already analytic (polynomial) functions. 
Another approach to regularity results similar
is the use of Malliavin calculus, see Gu\'erin \cite{gu}.\bla

\vip
 
Finally, let us mention the current theory for other Landau equations.  In the case of 
soft potentials  
$\gamma\in (-3, 0)$,  we refer to
\cite{fgue,fh}. The Coulomb case $\gamma=-3$, which is most directly physically relevant and 
significantly more difficult, has also received significant attention, including 
by Villani \cite{v:nc}, Desvillettes \cite{d} and the first author \cite{fc}. \bla 
Let us mention a number of works
by  Guo \cite{guo}, He-Yang \cite{hy}, Golse-Imbert-Mouhot-Vasseur \cite{go} and Mouhot
\cite{mou} on the Cauchy problem for the full, spatially inhomogeneous, Landau equation.
Finally, the Landau-Fermi-Dirac equation has been recently studied by Alonso, Bagland and 
Lods \cite{abl}. \bla

\subsection{Strategy}
We emphasise that the main result is the stability and uniqueness result Theorem \ref{main}; Theorem \ref{mainregularity} about regularity will then follow from previous works. \bla
For Theorem \ref{main}, our strategy is to build on the techniques of \cite{fgui}, \cite{n} and \cite{h2}. \bla The key new 
idea is a Povzner-type inequality \cite{p}.
Considering a weighted cost of the form $(1+|v|^p+|\tv|^p)|v-\tv|^2$ instead of
$|v-\tv|^2$, 
an additional, 
negative `Povzner term' arises which produces an advantageous cancelation and allows 
us to use a Gr\"onwall inequality. We rather study (at the price of technical difficulties)
the cost 
$(1+|v|^p+|\tv|^p)|v-\tv|^2/(1+|v-\tv|^2)$, because it requires less moments to be well-defined.

\vip

In the case of the Boltzmann equation for hard potentials without cutoff \cite{h2}, this technique leads to 
stability under the assumption only of some $p^\text{th}$ moment, for some computable, but potentially large, 
$p$, improving on previous results which required exponential moments \cite{fm}. In the case of the Landau 
equation,
the calculations become more tractable; we find explicit, rather than explicit\emph{able} constants, and 
are able to use tricks of \cite{fgui} and \cite{fh}. In this context, we seek to minimise the number $p$ of 
moments required, and very delicate calculations are needed, see Lemma \ref{cent} and its proof,
to allow for any $p>2$.

\subsection{Plan of the Paper}

The paper is structured as follows.  In Section \ref{pre}, we will present some preliminary calculations which are used throughout the paper. In Section \ref{moments}, we will prove some useful moment properties, including Proposition \ref{expo}. \vip 

Section \ref{coupling} - \ref{proof of main} are devoted to the proof of our stability result Theorem \ref{main}. Section \ref{coupling} introduces the Tanaka-style coupling and presents the key estimate without proof. This allows us to prove Theorem \ref{main} in Section \ref{proof of main}, and we finally return to prove the central estimate in Section \ref{proof of cent}. Informally, the main important points of the proof are Proposition \ref{coup},
where we introduce the coupling between two given weak solutions, 
Lemmas \ref{ito} and \ref{cent}, containing the central computation, and
Lemma \ref{mainexp} where we establish the stability estimate \eqref{eq: conclusion of main} 
under some additional conditions. Since the proof of the central computation Lemma \ref{cent} is rather technical, it is deferred until Section \ref{proof of cent}.

\vip 
Section \ref{existence} consists of a self-contained proof of our existence result 
Theorem \ref{mainexist}, building only on Theorem \ref{ddv} and using the de La Vall\'ee 
Poussin theorem and a compactness argument.

\vip 
In Section \ref{pf of regularity}, we prove Theorem \ref{mainregularity} about smoothness.
We show a very mild regularity result (Lemma \ref{weak regularity}): 
solutions do not remain concentrated on lines.  This allows us to apply 
Theorems \ref{ddv}-(c) and \ref{analytic regularity}, exploiting the uniqueness from 
Theorem \ref{main}.  
\vip 
Finally, Section \ref{proof of cent} contains the proof of the estimate Lemma \ref{cent}.

\subsection*{Acknowledgements} The second author is supported by the UK Engineering and Physical Sciences Research Council (EPSRC) grant EP/L016516/1 for the University of Cambridge Centre for Doctoral Training, the Cambridge Centre for Analysis. 

\section{Preliminaries}\label{pre}

We introduce a few notation and handle some computations of constant use.
We denote by $| \cdot |$ the Euclidean norm on $\rd$ and for
$A$ and $B$ two $3\times 3$ matrices, we put $\| A \|^2 = \Tr (AA^*)$ and $\lps A,B \rps=\Tr (A B^*)$.

\subsection{A few estimates of the parameters of the Landau equation}

For $x\in \rd$, we introduce 
$$\sigma(x)=[a(x)]^{1/2}=|x|^{1+\gamma/2}\Pi_{x^\perp}.$$ 
For $x,\tx \in \rd$, it holds that
\begin{equation}\label{p1}
||\sigma(x)||^2=2|x|^{\gamma+2} \;\;\; \hbox{and}\;\;\;
\lps \sigma(x),\sigma(\tx)\rps =|x|^{1+\gamma/2}|\tx|^{1+\gamma/2}\Big(1+ \frac{(x\cdot \tx)^2}{|x|^2|\tx|^2} \Big)
\geq 2 |x|^{\gamma/2}|\tx|^{\gamma/2}(x\cdot \tx).
\end{equation}
Indeed, it suffices to justify the second assertion, and a simple computation shows that
$\Pi_{x^\perp}\Pi_{\tx^\perp}=\Id - |x|^{-2}xx^*-|\tx|^{-2}\tx\tx^*+|x|^{-2}|\tx|^{-2}(x\cdot\tx)x\tx^*$, from
which we conclude that
\begin{align*}
\lps \sigma(x),\sigma(\tx)\rps=&|x|^{1+\gamma/2}|\tx|^{1+\gamma/2}\Tr \;
(\Pi_{x^\perp}\Pi_{\tx^\perp}) 
=|x|^{1+\gamma/2}|\tx|^{1+\gamma/2}[1+|x|^{-2}|\tx|^{-2}(x\cdot\tx)^2],
\end{align*}
which is greater than 
$2 |x|^{\gamma/2}|\tx|^{\gamma/2}(x\cdot \tx)$ because $1+a^{2}\geq 2a$.

\vip

For $a,b\geq 0$ and $\alpha \in (0,1)$, there holds
\begin{equation}\label{ttaacc}
|a^\alpha-b^{\alpha}| \leq (a\lor b)^{\alpha-1}|a-b|.
\end{equation}
Indeed, if e.g. $a\geq b$, then $a^\alpha-b^{\alpha}=a^{\alpha}[1-(b/a)^{\alpha}]\leq a^{\alpha}(1-b/a)=a^{\alpha-1}(a-b)$.

\vip

For $x,\tx \in \rd$, recalling that $b(x)=-2|x|^\gamma x$, \bla we have
\begin{equation}\label{p2}
|b(x)-b(\tx)|\leq 2|x|^\gamma|x-\tx|+2|\tx| ||x|^\gamma-|\tx|^\gamma| \leq 2(|x|^\gamma+|\tx|^\gamma)|x-\tx|,
\end{equation}
because $|\tx| ||x|^\gamma-|\tx|^\gamma|\leq |\tx| (|x|\lor|\tx|)^{\gamma-1}|x-\tx|\leq |\tx|^\gamma|x-\tx|$ by 
\eqref{ttaacc}. We also have, thanks to \eqref{p1},
\begin{equation}\label{p3}
||\sigma(x)-\sigma(\tx)||^2\leq 2|x|^{\gamma+2}+2|\tx|^{\gamma+2}-4|x|^{\gamma/2}|\tx|^{\gamma/2}(x\cdot\tx)
=2||x|^{\gamma/2}x-|\tx|^{\gamma/2}\tx|^2.
\end{equation}
Proceeding as for \eqref{p2}, we deduce that
\begin{equation}\label{p4}
||\sigma(x)-\sigma(\tx)||^2\leq 2 (|x|^{\gamma/2}|x-\tx|+|\tx|||x|^{\gamma/2}-|\tx|^{\gamma/2}|)^2
\leq   2(|x|^{\gamma/2}+|\tx|^{\gamma/2}|)^2|x-\tx|^2.
\end{equation}
Finally, for $v,\vs \in \rd$, $\sigma(v-\vs)v=\sigma(v-\vs)\vs$, because $\Pi_{(v-\vs)^\perp}(v-\vs)=0$, and so 
\begin{equation}\label{tr}
|\sigma(v-\vs)v|\leq C||\sigma(v-\vs)|| (|v|\land|\vs|)\leq C |v-\vs|^{1+\gamma/2} (|v|\land|\vs|)
\leq C |v-\vs|^{\gamma/2}|v||\vs|,
\end{equation}
because $|v-\vs|(|v|\land|\vs|)\leq (|v|+|\vs|)(|v|\land|\vs|)\leq 2 |v||\vs|$.

\subsection{Transport costs}
For technical reasons, we will have to play with a larger family of transport costs. For $p>0$
and $\e> 0$, for $f,\tf\in\cP_p(\rd)$, we define
$$
\cT_{p,\e}(f,\tf)= \inf \Big\{\intrdrd c_{p,\e}(v,\tv)R(\dd v,\dd \tv) \; : \; R \in \cH(f,\tf)   \Big\},
$$
where 
\begin{equation}\label{cpe}
c_{p,\e}(v,\tv)=(1+|v|^p+|\tv|^p)\varphi_\e(|v-\tv|^2) \quad \hbox{and}\quad \varphi_\e(r)=\frac r {1+\e r}.
\end{equation}  
We have $\cT_{p}=\cT_{p,1}$.  This definition also makes sense in the case $\e=0, \phi_0(r)=r$, in which case we require $f, \tf \in \cP_{p+2}(\rd)$ for the integral to be well-defined. In either case, it is straightforward to see that there exists a coupling attaining the infimum; we refer to Villani \cite{v: ot} for many details of such costs. \bla  Since
$\varphi'_\e(r)=(1+\e r)^{-2}$ and $\varphi''_\e(r)=-2\e(1+\e r)^{-3}$,
\begin{align}\label{ve}
r\varphi_\e'(r)\leq \varphi_\e(r),\quad 0\leq \varphi_\e'(r)\leq 1 \quad \hbox{and}\quad \varphi_\e''(r)\leq 0.
\end{align}

Let us remark that the cost $c_{p,\e}$ satisfies a relaxed triangle inequality: for some $C>0$ depending
only on $p>0$ and $\e\geq 0$, for all $v,w,y\in \rd$, 
\begin{equation}\label{eq: rti}c_{p,\e}(v,y)\le C [c_{p,\e}(v,w)+c_{p,\e}(w,y) ] \bla.
\end{equation} 
 The case where $\e=0$ was treated in \cite[Section 2]{h2}. If now $\e>0$,
${\frac{1}{2}(\e^{-1}\land r)}\le \varphi_\e(r)\le (\e^{-1}\land r)$,
so that it suffices to prove \eqref{eq: rti} with the cost $c_{p,\e}(v,\tv)$ 
replaced by $(1+|v|^p+|\tv|^p)(|v-\tv|^2\land 
\e^{-1})$.
This can be deduced from the case where $\e=0$, case-by-case, depending on which of $|v-w|^2$, $|w-y|^2$, 
$|v-y|^2$ are less than $\e^{-1}$. 
\vip   

It follows that the optimal transportation costs $\cT_{p,\e}$ 
are semimetrics in that one replaces the usual triangle inequality with the bound, for all $f,g,h\in \cP_p$,
\begin{equation} \label{eq: rti2} \cT_{p,\e}(f,h)\le C [\cT_{p,\e}(f,g)+\cT_{p,\e}(g,h)].
\end{equation}

\section{Moment Properties of the Landau Equation}\label{moments}

This section is devoted to some moment estimates.
 We start with the appearance of Gaussian moments, following the strategy introduced by Bobylev \cite{b}
for the Boltzmann equation.

\begin{proof}[Proof of Proposition \ref{expo}]
We consider any weak solution $(f_t)_{t\geq 0}$ to \eqref{LE}. \bla
By Theorem \ref{ddv}, we know that $m_2(f_t)=m_2(f_0)$ for all $t\geq 0$.
If $m_2(f_0)=0$, we deduce that $f_t=\delta_0$ for all $t>0$, so the result is obvious.
We thus assume that $m_2(f_0)>0$ and, by scaling, that $m_2(f_0)=1$. During the proof, 
$C$ will denote a constant which may only depend on $\gamma$, but may vary from line to line.

\vip

{\bf Step 1.} Here we prove that for all $p\geq 2$, all $t>0$,

$$
\frac d{dt} m_p(f_t) \leq -p m_{p+\gamma}(f_t) + p m_p(f_t) + C p^2[m_{p-2+\gamma}(f_t)+m_{p-2}(f_t)
m_{2+\gamma}(f_t)].
$$

By Theorem \ref{ddv}, we know that for all $q>0$, all $t_0>0$, $\sup_{t\geq t_0}m_q(f_t)<\infty$,
so that we can apply \eqref{wf} with $\varphi(v)=|v|^p$ on $[t_0,\infty)$.
We deduce that $m_p(f_t)$ is of class $C^1$ on $(0,\infty)$ and get
\begin{align}\label{dmp}
\frac d{dt} m_p(f_t)= \intrd\intrd \cL \varphi(v,\vs) f_t(\dd \vs) f_t(\dd v) \quad
\hbox{for all $t>0$.}
\end{align}
Since \bla $\varphi(v)=|v|^p=(v_1^2+v_2^2+v_3^2)^{p/2}$,  we have 
$$
\partial_{k}\varphi(v)=p|v|^{p-2}v_k \quad \hbox{and}\quad \partial^2_{k\ell}\varphi(v)=p|v|^{p-2}\indiq_{\{k=\ell\}}+
p(p-2)|v|^{p-4}v_kv_\ell.
$$
We set $x=v-\vs$
and note that, since $\sigma(x)=[a(x)]^{1/2}$ is symmetric,
$\sum_{k,\ell=1}^3 a_{k\ell}(x)\indiq_{\{k=\ell\}} = \Tr \; a(x)=
||\sigma(x)||^2$ and 
$\sum_{k,\ell=1}^3 a_{k\ell}(x)v_kv_\ell = \sum_{k,\ell,j=1}^3 \sigma_{kj}(x)\sigma_{\ell j}(x)v_kv_\ell
=|\sigma(x)v|^2 $. Thus \bla
\begin{align}\label{tto}
\cL\varphi(v,v_*)=p|v|^{p-2}v\cdot b(x)+\frac p2|v|^{p-2}||\sigma(x)||^2
+\frac{p(p-2)}2|v|^{p-4}|\sigma(x)v|^2.
\end{align}
Recalling that $b(x)=-2|x|^\gamma x$ and that $||\sigma(x)||^2=2|x|^{\gamma+2}$ by \eqref{p1}, 
\bla we have 
$$
v\cdot b(x) + \frac12||\sigma(x)||^2=-2|x|^\gamma(v-\vs)\cdot v +|x|^\gamma (|v|^2+|\vs|^2 -2v\cdot\vs)
=-|x|^\gamma |v|^2+|x|^\gamma |\vs|^2.
$$
Since moreover $|\sigma(x)v|\leq C |x|^{\gamma/2}|v||\vs|$ by \eqref{tr}, we find that
\begin{align*}
\cL\varphi(v,v_*)\leq -p|x|^\gamma |v|^p 
+ C p^2 |x|^\gamma |v|^{p-2}|\vs|^2.
\end{align*}
Using now that $|x|^\gamma \geq |v|^\gamma -|\vs|^\gamma$ and that 
$|x|^\gamma \leq |v|^\gamma +|\vs|^\gamma$, we conclude that

\begin{align}\label{tto2}
\cL\varphi(v,v_*)\leq& -p |v|^{p+\gamma} 
+p|v|^p|\vs|^\gamma + C p^2(|v|^{p-2+\gamma}|\vs|^2+|v|^{p-2}|\vs|^{2+\gamma}).
\end{align}
Plugging this into \eqref{dmp}, we find that 

\begin{align*}
\frac d{dt} m_p(f_t)\leq& -p m_{p+\gamma}(f_t) +pm_p(f_t) m_\gamma(f_t) 
+C p^2 (m_{p-2+\gamma}(f_t) m_2(f_t) + m_{p-2}(f_t) m_{2+\gamma}(f_t)  ).
\end{align*}
The conlusion follows, since  $m_\gamma(f_t)\leq [m_2(f_t)]^{\gamma/2}=1$.

\vip
{\bf Step 2.} We now deduce that for all $p\geq 4$,

$$
\frac d{dt} m_p(f_t) \leq -p [m_p(f_t)]^{1+\gamma/(p-2)} +  p m_p(f_t)+ 
C p^2 [m_p(f_t)]^{1-(2-\gamma)/(p-2)}.
$$
For any $\beta>\alpha\geq 2$, since $|v|^2f_t(\dd v)$ is a probability measure,
$$
m_\alpha(f_t)=\intrd |v|^{\alpha-2} |v|^2 f_t(\dd v) \leq 
\Big(\intrd |v|^{\beta-2} |v|^2 f_t(\dd v) \Big)^{(\alpha-2)/(\beta-2)}
=[m_\beta(f_t)]^{(\alpha-2)/(\beta-2)}.
$$
We deduce that $m_p(f_t) \leq [m_{p+\gamma}(f_t)]^{(p-2)/(p+\gamma-2)}$, whence
$$
m_{p+\gamma}(f_t) \geq [m_p(f_t)]^{(p+\gamma-2)/(p-2)}=[m_p(f_t)]^{1+\gamma/(p-2)},
$$
that
$$
m_{p-2+\gamma}(f_t) \leq [m_p(f_t)]^{(p-4+\gamma)/(p-2)},
$$
and that
$$
m_{p-2}(f_t) m_{2+\gamma}(f_t) \leq [m_p(f_t)]^{(p-4)/(p-2)+\gamma/(p-2)}=[m_p(f_t)]^{(p-4+\gamma)/(p-2)}.
$$
This completes the step, since $(p-4+\gamma)/(p-2)=1-(2-\gamma)/(p-2)$.
\vip
{\bf Step 3.} For $u:(0,\infty)\to(0,\infty)$ of class $C^1$ satisfying, 
for some $a,b,c,\alpha,\beta>0$, for all $t>0$,
$$
u'(t) \leq -a [u(t)]^{1+\alpha} +b u(t) + c[u(t)]^{1-\beta},
$$
it holds that
$$
\forall \; t>0, \quad u(t) \leq \Big(\frac{2}{a\alpha t} \Big)^{1/\alpha} + \Big(\frac{4b}{a} \Big)^{1/\alpha}
+\Big(\frac{4c}{a} \Big)^{1/(\alpha+\beta)}.
$$

Indeed, we set $h(r)=-a r^{1+\alpha} +b r + cr^{1-\beta}$ and we observe that 
$$
h(r) \leq -\frac a2 r^{1+\alpha} \quad \hbox{for all $r\geq u_*=\max\{(4b/a)^{1/\alpha},(4c/a)^{1/(\alpha+\beta)}\}$.}
$$
We now fix $t_0>0$.

\vip

(a) If $u(t_0)\leq u_*$, we have  $u(t) \leq u_*$ for all $t\geq t_0$
because $h(u_*)\leq 0$ and $u'(t)\leq h(u(t))$.

\vip

(b) If now $u(t_0)>u_*$, we set $t_1=\inf\{t>t_0 : u(t)\leq u_*\}$ and observe that
for $t\in [t_0,t_1)$,
$$u'(t) \leq h(u(t))\leq -\frac a2 [u(t)]^{1+\alpha}.$$ 
Integrating this inequality,
we conclude that, for all $t\in [t_0,t_1)$,
$$
u(t) \leq \Big[u^{-\alpha}(t_0)+ \frac{a \alpha (t-t_0)}2\Big]^{-1/\alpha} \leq \Big[\frac 2{a\alpha (t-t_0)}
\Big]^{1/\alpha}.
$$
This implies that $t_1$ is finite.
Since now $u(t_1)=u_*$ by definition, we deduce from (a) that $u(t)\leq u_*$ for all
$t\geq t_1$.

\vip

Hence in any case, for any $t_0>0$, any $t>t_0$, $u(t)\leq \max\{u_*,[2/(a\alpha(t-t_0))]^{1/\alpha}\}$.
Letting $t_0\to 0$, we deduce that $u(t)\leq \max\{u_*,[2/(a\alpha t)]^{1/\alpha}\}$
for all $t>0$, which completes the step.

\vip
{\bf Step 4.} Using Step 2 and applying Step 3 with 
$a=p$, $b=p$, $c=Cp^2$, $\alpha=\gamma/(p-2)$ and $\beta=(2-\gamma)/(p-2)$, 
we find that for all  $p\ge 4$, all $t>0$,
$$
m_p(f_t)\leq \Big(\frac{2(p-2)}{p\gamma t} \Big)^{(p-2)/\gamma} 
+ 4^{(p-2)/\gamma} \bla
+\Big(4Cp \Big)^{(p-2)/2}.
$$
Changing again the value of $C$, we conclude that for all $p\ge 4$, all $t>0$,
$$
m_p(f_t)\leq \Big(1+\frac{2}{\gamma t} \Big)^{p/\gamma} 
+(Cp)^{p/2}.
$$

{\bf Step 5.} For $a>0$ and $t>0$, we write, using that $m_0(f_t)=m_2(f_t)=1$,
$$
\intrd e^{a|v|^2}f_t(\dd v) = \sum_{k\geq 0} \frac{a^k m_{2k}(f_t)}{k!}
=1+a + \sum_{k\geq 2} \frac{a^k m_{2k}(f_t)}{k!}.
$$
By Step 4, 
$$
\intrd e^{a|v|^2}f_t(\dd v)
\leq 1+a+\sum_{k\geq 2} \frac 1{k!}\Big[a^k\Big(1+\frac{2}{\gamma t} \Big)^{2k/\gamma} 
+ a^k(2Ck)^{k}  \Big].
$$
But  $\sum_{k\geq 2} (k!)^{-1}(x k)^k <\infty$ if $x<1/e$ by the Stirling formula.
Hence if $a<1/(2Ce)$, 
$$
\intrd e^{a|v|^2}f_t(\dd v)
\leq 1+a+\exp\Big[a\Big(1+\frac{2}{\gamma t}\Big)^{2/\gamma}\Big]+C.
$$
The conclusion follows.
\end{proof}

We next prove some technical uniform integrability property.

\begin{lem}\label{ui} 
Fix $\gamma \in (0,1]$ and $p>2$. 
Let $(f_t)_{t\ge 0}$ be a weak solution to \eqref{LE}, with initial moment $m_p(f_0)<\infty$. 
Then, for all $\epsilon>0$, there exists $M<\infty$ such that 
$$ \limsup_{t\downarrow 0} \intrd (1+|v|^p) \indiq_{\{|v|>M\}}f_t(\dd v)<\epsilon.$$ 
\end{lem} 

\begin{proof} Let $\psi: \mathbb{R} \rightarrow [0,1]$ be a smooth function such that 
$\indiq_{\{r\leq 1\}} \leq \psi (r)\leq \indiq_{\{r \le 2\}}$. Now, for $M\ge 1$, define  
$\chi_M:\rd \rightarrow [0,1]$ 
by $\chi_M(v)=\psi(|v|/M)$; these functions are smooth, and satisfy 
\begin{equation*}
|v||\nabla \chi_M(v)|\le C; \qquad |v|^2|\nabla^2 \chi_M(v)|\le C
\end{equation*} 
for some constant $C$, independent of $M$.
A rough computation using that $|b(x)|\leq C|x|^{1+\gamma}$
and $||a(x)||\leq C |x|^{2+\gamma}$ shows that the smooth 
functions $\varphi_M(v)=(1+|v|^p)\chi_M(v)$ satisfy  
\begin{align*} 
|\cL \varphi_M(v,\vs)|\le& C[|b(v-\vs)| |\nabla\varphi_M(v)|+ ||a(v-\vs)|||\nabla^2\varphi_M(v)| ]
\\
\leq & C [|v-\vs|^{1+\gamma} (1+|v|^{p-1})+ |v-\vs|^{2+\gamma} (1+|v|^{p-2})]\\   
\leq& C (1+|v|^{p+\gamma}+|\vs|^{p+\gamma})
\end{align*} 
\bla for some $C$ which does not depend on $M$.   
It follows from \eqref{wf} \bla that, for all $M$, 
$$ \Big|\intrd \varphi_M(v)(f_t-f_0)(\dd v)\Big|\le C\intot m_{p+\gamma}(f_s)\dd s.$$ 
Now, fix $\epsilon>0$. Since $m_{p+\gamma}(f_s)$ is locally integrable by Theorem \ref{ddv}, 
there is $t_0>0$ such that for all $t\in [0,t_0]$ and all $M\ge 1$, 
\begin{equation}\label{ap1}
\Big|\intrd \varphi_M(v)(f_t-f_0)(\dd v)\Big|\le \frac\e3.
\end{equation}
We next fix $M\geq 1$ \bla such that 
\begin{equation}\label{ap2}
\intrd (1+|v|^p)\indiq_{\{|v|\geq M/2\}}f_0(\dd v)<\frac{\epsilon}{3}.
\end{equation}
For \bla any $t \in [0,t_0]$, any $M'\ge M$, 
\begin{align*}
\intrd (1&+|v|^p)\indiq_{\{M<|v|\le M'\}}f_t(\dd v)\le \intrd (\varphi_{M'}-\varphi_{M/2})(v)f_t(\dd v)\\
=& \intrd \varphi_{M'}(v)(f_t-f_0)(\dd v)-\intrd \varphi_{M/2}(v)(f_t-f_0)(\dd v) + 
\intrd (\varphi_{M'}-\varphi_{M/2})(v)f_0(\dd v) \leq \e.
\end{align*}
For the two first terms, we used \eqref{ap1}, while for the last term,  
we used that $(\varphi_{M'}-\varphi_{M/2})(v)\leq (1+|v|^p)\indiq_{\{|v|\geq M/2\}}$ and \eqref{ap2}.
Taking the limit $M'\rightarrow \infty$ now gives the result.

\end{proof}

\section{Tanaka-style Coupling of Landau Processes}\label{coupling}

In the spirit of Tanaka \cite{t} for the Boltzmann equation, see Funaki \cite{f} and Gu\'erin
\cite{g} for the Landau equation, we will use the following coupling between solutions. 
For $E=\rd$ or $\rd\times\rd$, we denote by $C^2_p(E)$ the set of $C^2$ functions on $E$
of which the derivatives of order $0$ to $2$ have at most polynomial growth.

\begin{prop}\label{coup}
Fix $\gamma \in (0,1]$, consider two weak solutions $(f_t)_{t\geq 0}$ and $(\tf_t)_{t\geq 0}$ to \eqref{LE} 
such that $\intrd e^{a |v|^2}(f_0+\tf_0)(\dd v)<\infty$ for some $a>0$, and fix $R_0 \in \cH(f_0,\tf_0)$. 
There exists a family $(R_t)_{t\geq 0}$ of probability measures on $\rd\times\rd$ such that
for all $t\geq 0$, $R_t \in \cH(f_t,\tf_t)$ and for all $\psi \in C^2_p(\rd\times\rd)$,
\begin{align}\label{wec}
\intrdrd \psi(v,\tv) R_t(\dd v,\dd \tv)=& \intrdrd \psi(v,\tv) R_0(\dd v,\dd \tv)\\
&+ \intot \intrdrd  \intrdrd  \cLL\psi(v,\vs,\tv,\tvs) R_s(\dd \vs,\dd \tvs) R_s(\dd v,\dd \tv)
\dd s,\notag
\end{align}
where
\begin{align*}
\cLL\psi(v,\tv,\vs,\tvs)=&\sum_{k=1}^3 [b_k(v-\vs) \partial_{v_k}\psi(v,\tv) + b_k(\tv-\tvs)
 \partial_{\tv_k} \psi(v,\tv) ] \\
&+\frac{1}{2}\sum_{k,\ell=1}^3 [a_{k\ell}(v-\vs)\partial^2_{v_kv_\ell}\psi(v,\tv)+ 
a_{k\ell}(\tv-\tvs)\partial^2_{\tv_k\tv_\ell}\psi(v,\tv)]\\
& +\sum_{j,k,\ell=1}^3 \sigma_{k j}(v-\vs)\sigma_{\ell j}(\tv-\tvs) \partial^2_{v_k \tv_\ell}\psi(v,\tv).
\end{align*}
\end{prop}

\begin{rk} Let us make the following observations. 
\vip
(i) This is the key coupling of $f_t, \tf_t$ which we will use,  for some
well-chosen $R_0$, to obtain an upper bound of 
$\cT_p(f_t, \tf_t)$ to prove Theorem \ref{main}. 
\vip
(ii) This equation has a natural probabilistic meaning:  the equation governing
$(R_t)_{t\geq 0}$ is the Kolmogorov equation for 
the solution $(V_t,\tV_t)_{t\geq 0}$ to the nonlinear stochastic differential equation 
\begin{equation*} \label{full sde}\begin{cases} V_t=V_0+\int_0^t \intrdrd b(V_s-\vs) 
R_s(\dd\vs,\dd\tvs)\dd s+ \int_0^t \intrdrd \sigma(V_s-\vs)N(\dd\vs, \dd\tvs,\dd s) ; \\ 
\tV_t=\tV_0+\int_0^t \intrdrd b(\tV_s-\tvs) R_s(\dd\vs,\dd\tvs)\dd s
+ \int_0^t \intrdrd \sigma(\tV_s-\tvs)N(\dd\vs, \dd\tvs,\dd s) ; \\ 
R_t=\Law(V_t,\tV_t) \end{cases} \end{equation*} 
where $N=(N^1, N^2,N^3)$ is a  $3D$-white noise 
on $\rd\times\rd\times[0,\infty)$ with covariance measure
$R_s(\dd \vs,\dd\tvs)\dd s$; 
see Walsh \cite{w}. We think of this nonlinear equation as describing the 
time evolution of the velocities $(V_t,\tV_t)_{t\geq0}$ 
of a `typical' pair of particles, with $V_t \sim f_t$ and $\tV_t \sim \tf_t$.
\vip
 (iii) Since $R_s \in \cH(f_s,\tf_s)$, we have $\intrdrd b(V_s-\vs) 
R_s(\dd\vs,\dd\tvs)\dd s=\intrd b(V_s-\vs) 
f_s(\dd\vs)\dd s$. Similarly,
$\int_0^t \intrdrd \sigma(V_s-\vs)N(\dd\vs, \dd\tvs,\dd s) =\int_0^t \intrd \sigma(V_s-\vs)
W(\dd\vs,\dd s)$, for some $3D$-white noise on $\rd\times[0,\infty)$ of  covariance measure 
$f_s(\dd \vs) \dd s$. Hence {\em in law}, the first SDE (for $(V_t)_{t\geq 0}$)
does not depend on $(\tf_t)_{t\geq 0}$.

\vip
(iv) The specific form of this coupling is important, rather than coupling processes using the 
same Brownian motion. The main idea is that we want $V_t$ and $\tV_t$ to be as close as possible.
Using the white noise in this way, we isolate the effect of a coupled pair 
$(\vs, \tvs)$,  with $\vs$ as close as possible to $\tvs$, in 
the background against our process $(V_s, \tV_s)$. It is also 
important that the white-noise covariance measure is $R_t(\dd \vs,\dd \tvs)\dd s$,
with $R_t$ the law of $(V_t,\tV_t)$. Replacing $R_t$, in the covariance measure  of the white noise,
with any other coupling (e.g. the optimal coupling for $\cT_p(f_t,\tf_t)$) \bla 
would not allow us
to use some symmetry arguments.

\vip
(v) We do not claim the uniqueness of solutions to \eqref{wec}; existence is 
sufficient for our needs.
\end{rk} 

\begin{proof}[Proof of Proposition \ref{coup}]
We sketch the proof, as the key points are standard for nonlinear diffusion equations and the 
Landau equation, see Gu\'erin \cite{g}. We fix $k\geq 1$ and define the truncated
{\it two level} coefficients 
$B_k:\rd\times\rd \rightarrow \rd\times\rd$ and $\Sigma_k: \rd\times\rd\rightarrow  {\mathcal M}_{6\times 3}(\rr)$ by 
$$ 
B_k\begin{pmatrix} x \\ \tx\end{pmatrix}=\begin{pmatrix} b_k(x) \\ b_k(\tx)\end{pmatrix};\qquad 
\Sigma_k \begin{pmatrix} x \\ \tx\end{pmatrix}=\begin{pmatrix} \sigma_k(x) \\ \sigma_k(\tx)\end{pmatrix}, 
$$ 
where $b_k(x)=-2(|x|\land k)^{\gamma} x$ and $\sigma_k(x)=(|x|\land k)^{\gamma/2} |x| \Pi_{x^\perp}$.
Proceeding as in \eqref{p2} and \eqref{p4}, 
one realises that $B_k$ and $\Sigma_k$ are globally Lipschitz continuous.

\vip

Now, let $W=(W^1, W^2, W^3)$ be a white noise on $[0,\infty)\times (0,1)$ with covariance measure
$\dd s\dd\alpha$. The usual arguments for 
nonlinear SDEs \cite{g} imply that there exists a process $X^k_t=(V^k_t, \tV^k_t)$ with initial distribution 
$X^k_0\sim R_0$, and a copy $Y^k_t$ defined on the probability space \bla
$((0,1), \mathcal{B}(0,1),d\alpha)$, 
with $\Law(X^k_t)=\Law(Y^k_t)$ and for all $t\geq 0$, \bla
\begin{equation*}
X^k_t=X^k_0+\int_0^t \int_{(0,1)} B_k(X^k_s-Y^k_s(\alpha))\hspace{0.1cm} \dd\alpha \dd s
+\int_0^t \int_{(0,1)} \Sigma_k(X^k_s-Y^k_s(\alpha)) W(\dd s,\dd \alpha).
\end{equation*} 
For $\psi \in C^2_p(\rd\times\rd)$,
applying It\^o's formula and taking expectations, we find
\begin{align*}
\E[\psi(X^k_t)]=&\E[\psi(X^k_0)]+ \intot\int_{(0,1)} \E[\nabla \psi(X^k_s)\cdot 
B_k(X^k_s-Y^k_s(\alpha))]  \dd\alpha \dd s\\
& +\frac12 \sum_{i,j=1}^6 \intot\int_{(0,1)} \E[\partial_{ij} \psi(X^k_s) 
[\Sigma_k(X^k_s-Y^k_s(\alpha)) \Sigma_k^*(X^k_s-Y^k_s(\alpha))]_{ij} ]\dd \alpha \dd s.
\end{align*}
Writing $R^k_t$ for the law of $X^k_t$ (and of $Y^k_t$), we thus get
\begin{align*}
\int_{\rd\times\rd}\!\! \psi(x) R_t^k(\dd x)=&\int_{\rd\times\rd} \!\!\psi(x) R_0(\dd x)+
\intot\!\! \intrdrd\! \intrdrd \!\!\nabla \psi(x)\cdot B_k(x-\xs) R_s^k(\dd x) R_s^k(\dd \xs)\dd s\\
&\!\!+\frac12 \sum_{i,j=1}^6 \intot\!\! \intrdrd\!\intrdrd \!\!\partial_{ij} \psi(x)
[\Sigma_k(x-\xs) \Sigma_k^*(x-\xs)]_{ij} ]R_s^k(\dd x) R_s^k(\dd \xs)\dd s.
\end{align*}
This precisely rewrites as 
\begin{align}\label{eq: approximate equations} 
\int_{\rd\times\rd} \psi(v,\tv) R_t^k(\dd v,\dd \tv)=&\int_{\rd\times\rd} \psi(v,\tv) R_0(\dd v,\dd \tv)\\
&+ \intot \intrdrd \intrdrd \cLL_k \psi(v,\tv,\vs,\tvs)
R_s^k(\dd v,\dd \tv) R_s^k(\dd \vs,\dd \tvs)\dd s,\notag
\end{align}
where $\cLL_k\psi$ is defined as $\cLL\psi$, replacing everywhere $b$, $\sigma$ and $a=\sigma\sigma^*$ 
by $b_k$, $\sigma_k$ and $a_k=\sigma_k\sigma^*_k$.

\vip

For $\psi(v,\tv)=\phi(v)+\phi(\tv)$, we have 
 $\cLL_k \psi(v,\tv,\vs,\tvs)=\cL_k\phi(v,\vs)+\cL_k\phi(\tv,\tvs)$,
where $\cL_k \phi$ is defined as $\cL\phi$, replacing $b$ and $a$ 
by $b_k$ and $a_k$. 
It is then straightforward to check that the approximate equation \eqref{eq: approximate equations} 
propagates \bla moments, uniformly in $k$, using arguments 
similar to those of \cite[Theorem 3]{dv1}
or Step 1 of the proof of Proposition \ref{expo}. 
In particular, under our initial Gaussian moment assumption,  
all moments of $R^k_t$ are bounded, uniformly $k\geq1$, locally uniformly in $t\geq 0$. 

\vip

It is then very classical to let $k\to \infty$ in \eqref{eq: approximate equations}, 
using a compactness argument, and to deduce
the existence of a family of probability measures $(R_t)_{t\geq 0}$ solving \eqref{wec} for all 
$\psi \in C^2_p(\rd\times\rd)$. See Section \ref{existence} 
for a similar procedure (with much less
moment estimates). \bla

\vip

Finally, we address the claim that $R_t$ is a coupling $R_t\in \cH(f_t, \tf_t)$. 
Let us write $g_t, \tilde{g}_t$ for the two marginals of $R_t$. For any $\varphi\in C^2_b(\rd)$,
we set $\psi(v,\tv)=\varphi(v)$ and observe that $\cLL\psi(v,\tv,\vs,\tvs)=\cL\varphi(v,\vs)$, so
that \eqref{wec} tells us that 
$$\intrd \varphi(v)g_t(\dd v)=\intrd \varphi(v)f_0(\dd v)
+\intot\intrd \cL\varphi(v,\vs) g_s(\dd\vs)g_s(\dd v)
\dd s.$$
\bla In other words, $(g_t)_{t\geq 0}$ is a weak solution to \eqref{LE}
which starts at $f_0$. Since $f_0$ is assumed to have a Gaussian moment, the uniqueness result 
Theorem \bla \ref{uniqueness} applies and so $(g_t)_{t\geq 0}=(f_t)_{t\geq 0}$ as desired. 
The argument that $(\tilde{g}_t)_{t\geq 0}=(\tf_t)_{t\geq 0}$ is identical.   
\end{proof}

We now carefully apply the coupling operator to our cost functions.

\begin{lem}\label{ito}
Adopt the notation of Proposition \ref{coup}
and fix $p\geq 2$ and $\e\in [0,1]$, and let $c_{p,\epsilon}$ be the transport cost defined in \eqref{cpe}. For $v,\vs,\tv,\tvs \in \rd$,
\begin{align*}
\cLL c_{p,\e}(v,\vs,\tv,\tvs) \leq& k_{p,\e}^{(1)}(v,\vs,\tv,\tvs)+k_{p,\e}^{(2)}(v,\vs,\tv,\tvs)
+k_{p,\e}^{(2)}(\tv,\tvs,v,\vs)\\
&+k_{p,\e}^{(3)}(v,\vs,\tv,\tvs)+k_{p,\e}^{(3)}(\tv,\tvs,v,\vs),
\end{align*}
where, setting $x=v-\vs$ and $\tx=\tv-\tvs$,
\begin{align*}
k_{p,\e}^{(1)}(v,\vs,\tv,\tvs)=& (1+|v|^p+|\tv|^p)\varphi_{\e}'(|v-\tv|^2)\Big[2(v-\tv)\cdot(b(x)-b(\tx))
+ ||\sigma(x)-\sigma(\tx)||^2   \Big],\\
k_{p,\e}^{(2)}(v,\vs,\tv,\tvs)=& \varphi_{\e}(|v-\tv|^2)\Big[p|v|^{p-2}v\cdot b(x)+\frac p2|v|^{p-2}||\sigma(x)||^2
+\frac{p(p-2)}2|v|^{p-4}|\sigma(x)v|^2\Big],
\\
k_{p,\e}^{(3)}(v,\vs,\tv,\tvs)=&2p |v|^{p-2}\varphi_{\e}'(|v-\tv|^2) [\sigma(x)v]\cdot[(\sigma(x)-\sigma(\tx))(v-\tv)].
\end{align*}
\end{lem}

\begin{proof} 
Fix $p\geq 2$, $\e\ge 0$ and let $\psi(v,\tv)=c_{p,\e}(v,\tv)=(1+|v|^p+|\tv|^p)\varphi_\e(|v-\tv|^2)$. 
We have
$$ \partial_{v_k}\psi(v,\tv)=p|v|^{p-2}v_k\varphi_\e(|v-\tv|^2)+2(v_k-\tv_k)(1+|v|^p+|\tv|^p)\varphi'_\e(|v-\tv|^2)$$ 
and a symmetric expression for $ \partial_{\tv_k}\psi(v,\tv)$.
Differentiating again, we find
\begin{align*}
\partial^2_{v_k v_\ell}\psi(v,\tv)=&p|v|^{p-2}\indiq_{\{k=\ell\}}\varphi_\e(|v-\tv|^2)+
 p(p-2)|v|^{p-4} v_kv_\ell\varphi_\e(|v-\tv|^2)  \\ 
&+ 2p|v|^{p-2}v_k(v_\ell-\tv_\ell)\varphi'_\e(|v-\tv|^2)+2\indiq_{\{k=\ell\}}(1+|v|^p+|\tv|^p)\varphi_\e'(|v-\tv|^2) \\ 
& + 4(v_k-\tv_k)(v_\ell-\tv_\ell)(1+|v|^p+|\tv|^p)\varphi''_\e(|v-\tv|^2) \\ 
& +2p|v|^{p-2}  (v_k-\tv_k)v_\ell   \varphi_\e'(|v-\tv|^2)
\end{align*} 
and a symmetric expression for $\partial^2_{\tv_k \tv_\ell}\psi(v,\tv)$. Concerning the cross terms,
\begin{align*} 
\partial^2_{v_k \tv_\ell}\psi(v,\tv)=&2p|v|^{p-2} v_k(\tv_\ell-v_\ell)\varphi'_\e(|v-\tv|^2)
+2p|\tv|^{p-2}  (v_k-\tv_k)\tv_\ell  \varphi'_\e(|v-\tv|^2)\\
&-4(v_k-\tv_k)(v_\ell-\tv_\ell)(1+|v|^p+|\tv|^p)\varphi_\e''(|v-\tv|^2) \\
& -2\indiq_{\{k=\ell\}}(1+|v|^p+|\tv|^p)\varphi'_\e(|v-\tv|^2).
\end{align*}

Let us now examine the sums in the definition of $\cLL \psi$ one by one. First,
\begin{align*}
&\sum_{k=1}^3  [b_k(v-\vs) \partial_{v_k}\psi(v,\tv) + b_k(\tv-\tvs) \partial_{\tv_k} \psi(v,\tv)] \\ 
=& p |v|^{p-2} v\cdot b(v-\vs)\varphi_\e(|v-\tv|^2)
& (=A_1) \\
&+ p |\tv|^{p-2} \tv\cdot b(\tv-\tvs)\varphi_\e(|v-\tv|^2)
& (=A_2) \\
& + 2(1+|v|^p+|\tv|^p)(v-\tv)\cdot (b(v-\vs)-b(\tv-\tvs))\varphi'_\e(|v-\tv|^2).& (=A_3)
\end{align*}
Next, using that for $x,y,z \in \rd$, $\Tr \; a(x)=||\sigma(x)||^2$ and 
$\sum_{k,\ell=1}^3 a_{k\ell}(x)y_kz_\ell= [\sigma(x)y]\cdot[\sigma(x)z]$,
\begin{align*}
\frac{1}{2}\sum_{k,\ell=1}^3 a_{k\ell}(v-\vs)&\partial^2_{v_kv_\ell}\psi(v,\tv) =
\frac{p}{2}|v|^{p-2}\|\sigma(v-\vs)\|^2\varphi_\e(|v-\tv|^2)  &(=B_1)\\
&+\frac{p(p-2)}{2}|v|^{p-4}|\sigma(v-\vs)v|^2\varphi_\e(|v-\tv|^2)  &(=B_2)\\ 
&+2p |v|^{p-2}[\sigma(v-\vs)v]\cdot[\sigma(v-\vs)(v-\tv)] \varphi'_\e(|v-\tv|^2) &(=B_3) \\ 
&+ (1+|v|^p+|\tv|^p)\|\sigma(v-\vs)\|^2\varphi'_\e(|v-\tv|^2) &(=B_4) \\ 
& +2 (1+|v|^p+|\tv|^p)|\sigma(v-\vs)(v-\tv)|^2\varphi''_\e(|v-\tv|^2). &(=B_5) 
\end{align*}
Similarly,
\begin{align*}
\frac{1}{2}\sum_{k,\ell=1}^3 a_{k\ell}(\tv-\tvs)&\partial^2_{\tv_k\tv_\ell}\psi(v,\tv) =  
\frac{p}{2}|\tv|^{p-2}\|\sigma(\tv-\tvs)\|^2\varphi_\e(|v-\tv|^2)&(=C_1)\\
&+\frac{p(p-2)}{2}|\tv|^{p-4}|\sigma(\tv-\tvs)\tv|^2\varphi_\e(|v-\tv|^2) &(=C_2)\\ 
&+2p|\tv|^{p-2}[\sigma(\tv-\tvs)\tv]\cdot[\sigma(\tv-\tvs)(\tv-v)]\varphi'_\e(|v-\tv|^2)&(=C_3)
 \\ 
&+ (1+|v|^p+|\tv|^p)\|\sigma(\tv-\tvs)\|^2\varphi'_\e(|v-\tv|^2)&(=C_4) \\ 
& +2 (1+|v|^p+|\tv|^p)|\sigma(\tv-\tvs)(\tv-v)|^2\varphi''_\e(|v-\tv|^2).&(=C_5) 
\end{align*}
Finally, we look at the cross-terms: 
\begin{align*}
&\sum_{j,k,\ell=1}^3 \sigma_{kj}(v-\vs)\sigma_{\ell j}(\tv-\tvs) \partial^2_{v_k\tv_\ell}\psi(v,\tv)&\\
=&  -2p |v|^{p-2}[\sigma(v-\vs)v]\cdot[\sigma(\tv-\tvs)(v-\tv)]\varphi'_\e(|v-\tv|^2) &(=D_1)\\ 
& +2p |\tv|^{p-2}[\sigma(v- \vs )(v-\tv)]\cdot[\sigma(\tv-\tvs)\tv]\varphi'_\e(|v-\tv|^2) &(=D_2)\\ 
&-4(1+|v|^p+|\tv|^p) [\sigma(v-\vs)(v-\tv)]\cdot [\sigma(\tv-\tvs)(v-\tv)]\varphi''_\e(|v-\tv|^2)&(=D_3)\\
& -2(1+|v|^p+|\tv|^p) \lps \sigma(v-\vs),\sigma(\tv-\tvs)\rps \varphi'_\e(|v-\tv|^2). &(=D_4)
\end{align*}
Recalling the notation $x=v-\vs$ and $\tx=\tv-\tvs$, we find that
\begin{align*}
A_3+B_4+C_4+D_4=&k_{p,\e}^{(1)}(v,\vs,\tv,\tvs),\\
A_1+B_1+B_2=&k_{p,\e}^{(2)}(v,\vs,\tv,\tvs),\\
A_2+C_1+C_2=&k_{p,\e}^{(2)}(\tv,\tvs,v,\vs),\\
B_3+D_1=&k_{p,\e}^{(3)}(v,\vs,\tv,\tvs),\\
C_3+D_2=&k_{p,\e}^{(3)}(\tv,\tvs,v,\vs),
\end{align*}
and finally that
\begin{align*}
B_5+C_5+D_3=&2 (1+|v|^p+|\tv|^p)|(\sigma(x)-\sigma(\tx))(v-\tv)|^2\varphi''_\e(|v-\tv|^2)\leq 0
\end{align*}
since $\varphi_\e''$ is nonpositive,  see \eqref{ve}. 
\end{proof}

We finally state the following central inequality.

\begin{lem}\label{cent}
 There is a constant $C$, depending only on $p\geq 2$ and $\gamma \in (0,1]$, 
such that for all $\e \in (0,1]$, all $v,\vs,\tv,\tvs \in \rd$,
\begin{align*}
\cLL c_{p,\e}(v,\vs,\tv,\tvs)
\leq& [2 - p] c_{p+\gamma,\e}(v,\tv)\\
&+ C\sqrt\e (1+|\vs|^p+|\tvs|^p) c_{p+\gamma,\e}(v,\tv)\\
&+ C\sqrt\e (1+|v|^p+|\tv|^p) c_{p+\gamma,\e}(\vs,\tvs)\\
&+ \frac{C}{\sqrt \e} (1+|\vs|^{p+\gamma}+|\tvs|^{p+\gamma})c_{p,\e}(v,\tv) \\
&+ \frac{C}{\sqrt \e} (1+|v|^{p+\gamma}+|\tv|^{p+\gamma})c_{p,\e}(\vs,\tvs).
\end{align*}
\end{lem}

Let us now highlight the main features of this bound, which motivate our strategy. 
The last two lines are amenable to a Gr\"onwall-type estimate, provided $\int_0^T m_{p+\gamma}(f_s+\tf_s)\dd s <\infty$, but this 
is prevented by the appearance of $c_{p+\gamma, \e}$ in the earlier terms; in
\cite{fgui}, analagous terms are handled using an exponential moment estimate. The key observation is that, 
by choosing $p>2$, the first line gives a negative multiple of this `bad' term, which can absorb the second 
and third lines if
$\epsilon>0$ is small enough (and if we know that $\sup_{[0,T]} m_{p}(f_s+\tf_s)<\infty$), 
allowing us to use a Gr\"onwall estimate. 

\vip

Let us mention that a rather direct computation, with $\e=0$, i.e. with the cost
$c_{p,0}(v,\tv)=(1+|v|^p+|\tv|^p)|v-\tv|^2$, relying on the simple estimates \eqref{p2}, \eqref{p4} and \eqref{tr},
shows that
\begin{align*}
\cLL c_{p,0}(v,\vs,\tv,\tvs)
\leq& [32 - p] c_{p+\gamma,0}(v,\tv)\\
&+C(1+|\vs|^{p+\gamma}+|\tvs|^{p+\gamma}) c_{p,0}(v,\tv)
+ C(1+|v|^{p+\gamma}+|\tv|^{p+\gamma}) c_{p,0}(\vs,\tvs).
\end{align*}
Choosing $p=32$, the first term is nonpositive, and this would lead to a 
stability result for the cost 
$\cT_{32,0}$, for initial conditions in $\cP_{34}(\rd)$, since $\cT_{32,0}$ requires some
moments of order $34$ to be well-defined. \bla

\vip

The proof of Lemma \ref{cent} is much more complicated; we have to be very careful and to use
many cancelations to replace $[32-p]$ by $[2-p]$. 
Moreover, we have to deal with $c_{p,\e}$ with $\e>0$ instead of 
$c_{p,0}$, because $\cT_{p,0}$ requires moments of order $p+2$ to be well-defined.
All this \bla is
crucial to obtain a stability result in  $\cP_{p}(\rd)$, for any $p>2$.
Since the proof is rather lengthy, it is deferred to Section \ref{proof of cent} 
for the ease of readability.

\section{Stability}\label{proof of main}

We now give the proof of our stability estimate. We first deal with the case when 
the initial data have a finite Gaussian moment, 
and then carefully relax this assumption.

\begin{lem}\label{mainexp} Fix $\gamma\in (0,1]$ and let $(f_t)_{t\ge 0}$, $(\tf_t)_{t\ge 0}$ be weak 
solutions to \eqref{LE} with initial moments 
$ \intrd e^{a|v|^2}(f_0+\tf_0)(\dd v)<\infty$ for some $a>0$. Then the stability estimate 
\eqref{eq: conclusion of main}  holds true.
\end{lem}
\begin{proof}
We fix $p>2$, consider $\e\in(0,1]$ to be chosen later and introduce 
$R_0\in\cH(f_0,\tf_0)$ such that
$$
\cT_{p,\e}(f_0,\tf_0)=\intrdrd c_{p,\e}(v,\tv) R_0(\dd v,\dd \tv).
$$
Note that $R_0$ depends on $\e$, but this is not an issue. We then introduce 
$(R_t)_{t\geq 0}$ as in Proposition \ref{coup}, which is licit thanks to our initial
Gaussian moment condition.
We know that for each $t\geq 0$, $R_t \in \cH(f_t,\tf_t)$,
from which we conclude that
\begin{equation}\label{ww}
u_\e(t)=\intrdrd c_{p,\e}(v,\tv) R_t(\dd v,\dd \tv) \geq
\cT_{p,\e}(f_t,\tf_t).
\end{equation}
By Proposition \ref{coup}, and since $u_\e(0)=\cT_{p,\e}(f_0,\tf_0)$, it holds that
 for all $t\geq 0$,
$$
u_\e(t)= \bla \cT_{p,\e}(f_0,\tf_0) + \intot \intrdrd\intrdrd \cLL c_{p,\e}(v,\vs,\tv,\tvs)
R_s(\dd \vs,\dd \tvs) R_s(\dd v,\dd \tv) \dd s.
$$
Using next Lemma \ref{cent} and a symmetry argument, we find that
$$
u_\e(t)\leq \cT_{p,\e}(f_0,\tf_0) + \intot (I_{1,\e}(s)+I_{2,\e}(s)+I_{3,\e}(s))\dd s,
$$
where, for some constant $C>0$ depending only on $p$ and $\gamma$,
\begin{align*}
I_{1,\e}(s)=& [2-p] \intrdrd c_{p+\gamma,\e}(v,\tv) R_s(\dd v,\dd \tv)
,\\
I_{2,\e}(s)=& C \sqrt \e \intrdrd\intrdrd (1+|\vs|^p+|\tvs|^p)c_{p+\gamma,\e}(v,\tv) 
R_s(\dd \vs,\dd \tvs) R_s(\dd v,\dd \tv),\\
I_{3,\e}(s)=& \frac{C}{\sqrt \e}\intrdrd\intrdrd
(1+|\vs|^{p+\gamma}+|\tvs|^{p+\gamma})c_{p,\e}(v,\tv) R_s(\dd \vs,\dd \tvs) R_s(\dd v,\dd \tv).
\end{align*}
Using that $R_s\in\cH(f_s,\tf_s)$, we conclude that
\begin{align*}
I_{2,\e}(s)\leq & C \sqrt \e (1+m_p(f_s+\tf_s))\intrdrd c_{p+\gamma,\e}(v,\tv)R_s(\dd v,\dd \tv),\\
I_{3,\e}(s)\leq & \frac C {\sqrt \e} (1+m_{p+\gamma}(f_s+\tf_s)) u_\e(s).
\end{align*}

We now fix $t>0$ and work on $[0,t]$. Setting $m_{p,\infty}([0,t])=
\sup_{s\in [0,t]} m_{p}(f_s+\tf_s)$ and choosing
$$
\e= \Big[ \frac{p-2}{p-2+C(1+m_{p,\infty}([0,t]))}\Big]^2,
$$
so that $\e \in (0,1]$ and
$2-p+C \sqrt \e (1+m_p(f_s+\tf_s)) \leq 0$ for all $s\in [0,t]$, we conclude
that $I_{1,\e}(s)+I_{2,\e}(s) \leq 0$ for all $s\in [0,t]$, whence
$$
u_\e(r)\leq \cT_{p,\e}(f_0,\tf_0) + \frac C {\sqrt \e}\int_0^r (1+m_{p+\gamma}(f_s+\tf_s))
u_\e(s) \dd s
$$
for all $r \in [0,t]$. The Gr\"onwall \bla lemma then tells us that
$$
\cT_{p,\e}(f_t,\tf_t)\leq  u_\e(t)\leq\cT_{p,\e}(f_0,\tf_0)
\exp \Big( \frac C {\sqrt \e}\intot (1+m_{p+\gamma}(f_s+\tf_s))\dd s \Big).
$$
Using finally that $\cT_p=\cT_{p,1}$ and that
$c_{p,1} \leq c_{p,\e} \leq \e^{-1} c_{p,1}$, we deduce that
$$
 \cT_p(f_t,\tf_t)\leq\cT_{p,\e}(f_t,\tf_t) \quad \hbox{and}\quad  
\cT_{p,\e}(f_0,\tf_0) \leq \frac 1\e \cT_{p}(f_0,\tf_0). 
$$
We thus end with
$$
\cT_{p}(f_t,\tf_t)\leq  \frac 1 \e \cT_{p}(f_0,\tf_0)
\exp \Big( \frac C {\sqrt \e}\intot (1+m_{p+\gamma}(f_s+\tf_s))\dd s \Big).
$$
Recalling our choice for $\e$ and allowing the value of $C$, still depending only
on $p$ and $\gamma$, to change from line to line, we find that
\begin{align*}
\cT_{p}(f_t,\tf_t)\leq&  C(1+m_{p,\infty}([0,t]))^2 \cT_{p}(f_0,\tf_0)
\exp \Big( C[1+m_{p,\infty}([0,t])]\intot (1+m_{p+\gamma}(f_s+\tf_s))\dd s \Big)\\
\leq&\cT_{p}(f_0,\tf_0)
\exp \Big( C[1+m_{p,\infty}([0,t])]\Big[1+\intot (1+m_{p+\gamma}(f_s+\tf_s))\dd s\Big] \Big),
\end{align*}
which was our goal.
\end{proof}

In order to relax the initial Gaussian moment condition, we 
will use the following convergence. 

\begin{lem}\label{tp convergence} 
Fix $\gamma \in (0,1]$ and $p>2$.
Let $(f_t)_{t\ge 0}$ be a weak solution to \eqref{LE}, with initial moment $m_p(f_0)<\infty$. 
Then $\cT_p(f_t, f_0)\rightarrow 0$ as $t\rightarrow 0$.
\end{lem}

\begin{proof} First, thanks to the density of $C^2_b(\rd)$  in $C_b(\rd)$, 
we deduce  from \eqref{wf} that $f_t\rightarrow f_0$ weakly. 
It classically follows that $\lim_{t\to 0}d(f_t,f_0)=0$, where $d$ is the following
distance that classicaly metrises weak convergence on probability measures:
\begin{equation*}
d(f,g)
 =\inf\Big\{\intrdrd (1\land |v-w|) S(\dd v,\dd w): S\in \cH(f, g) \Big\}. 
\end{equation*} 
Moreover, for each $t\geq 0$, there exists a coupling $S_t \in \cH(f_t,f_0)$ attaing the minimum
$d(f_t,f_0)= \intrdrd (1\land |v-w|) S_t(\dd v,\dd w)$. \bla
Now, fix $\epsilon>0$; by Lemma \ref{ui}, there exist $M<\infty$ and $t_0>0$ such that
$$
\intrd (1+|v|^p)\indiq_{\{|v|>M\}} f_t(\dd v)<\epsilon \quad \hbox{for all $t\in [0,t_0]$}.
$$
Since now $c_{p,1}(v,w) \leq (1+|v|^p+|w|^p)(|v-w|\land 1)
\leq (1+|v|^p)(|v-w|\land 1)+ (1+|w|^p)(|v-w|\land 1)$ and since 
$\cT_p=\cT_{p,1}$, we have
\begin{align*}
\cT_p(f_t, f_0) \leq & \intrdrd c_{p,1}(v,w)S_t(\dd v,\dd w) \\
\leq & (1+M^p) d(f_t,f_0)+ \intrdrd (1+|v|^p)\indiq_{\{|v|>M\}}S_t(\dd v,\dd w)\\
&+(1+M^p) d(f_t,f_0)+ \intrdrd (1+|w|^p)\indiq_{\{|w|>M\}}S_t(\dd v,\dd w)\\
=& 2(1+M^p)d(f_t,f_0)+ \intrd (1+|v|^p)\indiq_{\{|v|>M\}} f_t(\dd v)+\intrd (1+|w|^p)\indiq_{\{|w|>M\}} f_0(\dd w),
\end{align*}
the last equality using that $S_t \in \cH(f_t,f_0)$. We conclude that for all $t\in [0,t_0]$,
$$
\cT_p(f_t, f_0) \leq 2(1+M^p)d(f_t,f_0)+ 2\e,
$$
whence $\limsup_{t\to 0} \cT_p(f_t, f_0) \leq 2\e$ and we are done, as $\epsilon>0$ was arbitrary.
\end{proof}

We are now ready to remove the additional assumptions and prove the full stability statement.

\begin{proof}[Proof of Theorem \ref{main}] 
We fix $\gamma \in (0,1]$, $p>2\bla$ and we consider two weak solutions $(f_t)_{t\geq 0}$
and $(\tf_t)_{t\geq 0}$ to \eqref{LE} such that $m_p(f_0+\tf_0)<\infty$.

\vip

Fix $t>0$ and let $0<s\le t$; thanks to Proposition \ref{expo}, we have 
$\intrd e^{a|v|^2}(f_s+\tf_s)(\dd v)<\infty$ for some $a>0$. Lemma \ref{mainexp}
therefore applies to $(f_u)_{u\ge s}, (\tf_u)_{u\ge s}$, so that,
setting $m_{p,\infty}([s,t])=\sup_{r\in [s,t]}m_p(f_r+\tf_r)$,
\begin{align}\label{conclustion st}
\cT_p(f_t, \tf_t)\le &  
\cT_p(f_s,\tf_s)\exp \Big( C[1+m_{p,\infty}([s,t])]
\Big[1+\int_s^t (1+m_{p+\gamma}(f_u+\tf_u))\dd u\Big] \Big)\\
\leq &\cT_p(f_s,\tf_s)\exp \Big( C[1+m_{p,\infty}([0,t])]
\Big[1+\int_0^t (1+m_{p+\gamma}(f_u+\tf_u))\dd u\Big] \Big). \notag
\end{align} 
Recalling the relaxed triangle inequality \eqref{eq: rti2}, we have, for some constant 
$C$ depending only on $p$, 
$$\cT_p(f_s, \tf_s)\le C[\cT_p(f_s, f_0)+\cT_p(f_0, \tf_0)+\cT_p(\tf_0, \tf_s)] $$ 
and as $s\rightarrow 0$, the first and third terms converge to $0$ by Lemma \ref{tp convergence}, 
so 
$$ \limsup_{s\rightarrow 0} \cT_p(f_s, \tf_s)\le C\cT_p(f_0, \tf_0).$$ 
We thus can take $s\downarrow 0$ in \eqref{conclustion st} to obtain the desired result. 
\end{proof}

\section{Existence}\label{existence} 

\begin{proof}[Proof of Theorem \ref{mainexist}] 
Let us start from $f_0\in \cP_2$. By the de La Vall\'ee Poussin theorem, there exists a $C^2$-function 
$h:[0,\infty)\rightarrow [0,\infty)$ such that $h'' \geq 0$, $h'(\infty)=\infty$ and
\begin{equation} \label{eq: h integrable} 
\intrd h(|v|^2) f_0(\dd v)<\infty .
\end{equation}
We can also impose that $h''\le 1$ and that $h'(0)=1$. 

\vip

{\bf Step 1.}
We consider $n_0\geq 1$ such that for all $n\geq n_0$, $\alpha_n=
\intrd \indiq_{\{|v|\le n\}}f_0(\dd v) \geq 1/2$ and set, for $n\geq n_0$,
$$ f^n_0(\dd v)=\alpha_n^{-1}\indiq_{\{|v|\le n\}}f_0(\dd v) \in \cP(\rd).$$
Since $f^n_0$ is compactly supported, it has all
moments finite and there exists a weak solution $(f^n_t)_{t\geq 0}$ to \eqref{LE} starting at
$f^n_0$ by Theorem \ref{ddv}.
Of course, $f^n_0$ converges weakly to $f_0$ as $n\to \infty$.

{\bf Step 2.} We now show that for all $T>0$, there is a finite constant $K_T$ such that for all
$n\geq n_0$,
\begin{equation} \label{eq: UI in existence proof}
\sup_{t\in [0,T]} \intrd h(|v|^2)f^n_t (\dd v) + \int_0^T \intrd |v|^{2+\gamma}h'(|v|^2)f^n_t (\dd v) \dd t \leq K_T.\end{equation}

By Theorem \ref{ddv}, all polynomial moments of $f^n_t$ are bounded, uniformly in $t\geq 0$ 
(but not necessarily in 
$n$). We can therefore apply \eqref{wf} to the function $\varphi(v)=h(|v|^2)$: 
arguing as in \eqref{tto}, 
$$ \partial_k\varphi(v)=2v_kh'(|v|^2); \qquad \partial^2_{k\ell}\varphi(|v|^2)
=2h'(|v|^2)\indiq_{\{k=\ell\}}+4v_kv_lh''(|v|^2)$$  
and so, setting $x=v-\vs$ as usual,
$$ \cL \varphi(v,\vs)=h'(|v|^2)[2v\cdot b(x)+\|\sigma(x)\|^2]+2|\sigma(x)v|^2h''(|v|^2).  $$ 
Recalling \eqref{tr} and that $0\leq h''\leq 1$, \bla the last term is bounded by 
$$ 2|\sigma(x)v|^2h''(|v|^2) \le C |x|^\gamma |v|^2|\vs|^2 \bla
\le C(|v|^{2+\gamma}|\vs|^2+ |v|^2|\vs|^{2+\gamma}). $$ 
Meanwhile, since $b(x)=-2|x|^\gamma x$ and $||\sigma(x)||^2=2 |x|^{\gamma+2}$,
the first term is 
\begin{align*} h'(|v|^2)&[2v\cdot b(x)+\|\sigma(x)\|^2] 
 =2h'(|v|^2)[-|x|^\gamma|v|^2+|x|^\gamma|\vs|^2] \\
&\le -2h'(|v|^2)|v|^{2+\gamma} +2h'(|v|^2)|\vs|^\gamma|v|^2
+2h'(|v|^2)|v|^\gamma|\vs|^2 + 2h'(|v|^2)|\vs|^{2+\gamma}\\
&\le -h'(|v|^2)|v|^{2+\gamma} +C (1+|v|^2) |\vs|^{\gamma+2}.
\end{align*}
We used that $|x|^\gamma \geq |v|^\gamma -|\vs|^\gamma$, that $|x|^\gamma \leq |v|^\gamma -|\vs|^\gamma$ and, 
for the last inequality, that there is $C>0$ such that
$|\vs|^\gamma|v|^2+|v|^\gamma|\vs|^2 \leq \frac12|v|^{2+\gamma}+C|\vs|^{2+\gamma}$
and that $h'(r)\leq 1+r$. All in all,
$$
\cL \varphi(v,\vs) \leq -h'(|v|^2)|v|^{2+\gamma} +C (1+|v|^2) |\vs|^{\gamma+2}+C(1+|\vs|^2) |v|^{\gamma+2}.
$$
We thus find, by \eqref{wf}, recalling that $m_2(f^n_t)=m_2(f^n_0)$, that
\begin{align*}
&\intrd h(|v|^2)f^n_t(\dd v) + \intot \intrd h'(|v|^2)|v|^{2+\gamma} f^n_s(\dd v) \dd s\\
\leq&\intrd h(|v|^2)f^n_0(\dd v)
+2C(1+m_2(f_0^n))\intot \intrd |v|^{2+\gamma} f^n_s(\dd v) \dd s\\
\leq & 2\intrd h(|v|^2)f_0(\dd v)
+2C(1+2m_2(f_0))\intot \intrd |v|^{2+\gamma} f^n_s(\dd v) \dd s,
\end{align*}
since $f^n_0 \leq 2f_0$. But since $h'(\infty)=\infty$, there is a constant $\kappa$ (depending on $m_2(f_0)$)
such that $2C(1+2m_2(f_0)) |v|^{2+\gamma} \leq \frac 12 h'(|v|^2)|v|^{2+\gamma} + \kappa$ for all $v\in \rd$.
We finally get
$$
\intrd h(|v|^2)f^n_t(\dd v)+ \frac12 \intot \intrd h'(|v|^2)|v|^{2+\gamma} f^n_s(\dd v) \dd s
\leq 2\intrd h(|v|^2)f_0(\dd v) + \kappa t,
$$
and this completes the step.

\vip

{\bf Step 3.} Here we show that the family $((f^n_t)_{t\geq 0})_{n\geq n_0}$ is relatively compact 
in $C([0,\infty),\cP(\rd))$, where $\cP(\rd)$ is endowed with the usual weak convergence.
This last convergence can be metrised by the distance on $\cP(\rd)$:
$$
\delta(f,g)=\sup_{\varphi \in C^2_{b,1}}\Big|\intrd \varphi(v)(f-g)(\dd v)\Big|,
$$
where $C^2_{b,1}$ is the set of $C^2$ functions on $\rd$ such that $||\varphi||_\infty+||\nabla \varphi||_\infty
+||\nabla^2 \varphi||_\infty \leq 1$.
By the Arzel\`a-Ascoli theorem, it suffices to check that 

\vip

\noindent (a) for all $t\geq 0$, the family
$(f^n_t)_{n\geq n_0}$ is relatively compact in $\cP(\rd)$ and 

\vip
\noindent (b) for all $T>0$, 
$\lim_{\e\to 0} \sup_{n\geq n_0} \sup_{s,t \in [0,T], |t-s|\leq \e} \delta(f^n_t,f^n_s) = 0$.

\vip

Point (a) is obvious, since for all $t\geq 0$, all $n\geq n_0$, $m_2(f^n_t)\leq 2 m_2(f_0)$
and since the set $\{f \in \cP(\rd) : m_2(f) \leq a\}$ is compact for any $a>0$.
Concerning point (b), we recall that there is a constant $C$ such that for all $\varphi \in C^2_{b,1}$,
$|\cL\varphi(v,\vs)|\leq C(1+|v|^{\gamma+2}+|\vs|^{\gamma+2})$. We thus deduce from \eqref{wf} 
that for all $t\geq s \geq 0$, all
$n\geq n_0$,
$$
\delta(f^n_t,f^n_s) \leq C \int_s^t \intrdrd (1+|v|^{\gamma+2}+|\vs|^{\gamma+2}) f^n_s(\dd \vs)f^n_s(\dd v)\dd s
\leq 2C \int_s^t \intrd (1+|v|^{\gamma+2})f^n_s(\dd v)\dd s.
$$
Now for $0\leq s \leq t \leq T$ with $t-s\leq \e$, for any $n\geq n_0$, any $A>0$, 
separating the cases $|v|\leq A$ and $|v|\geq A$, 
\begin{align*}
\delta(f^n_t,f^n_s) \leq& 2C (1+A^{\gamma+2})(t-s) + \frac{2C}{h'(A^2)} 
\int_s^t \intrd (1+|v|^{\gamma+2})h'(|v|^2)
f^n_s(\dd v)\dd s \\
\leq& 2C (1+A^{\gamma+2}) \e + \frac{2CK_T}{h'(A^2)}
\end{align*}
because $h'$ is nondecreasing and with $K_T$ introduced in Step 2. 
\bla Now for $\eta>0$ fixed,
we choose $A_\eta>0$ large enough so that $\frac{2CK_T}{h'(A_\eta^2)} \leq \frac \eta 2$ and
conclude that, as soon as $\e \leq  \frac{\eta}{4C (1+A_\eta^{\gamma+2})}$, \bla we have
$\delta(f^n_t,f^n_s) \leq \eta$
for all $n\geq n_0$ and all $s,t \in [0,T]$ such that $|t-s|\leq \e$.

\vip

{\bf Step 4.} By Step 3, we can find a (not relabelled) subsequence such that $(f^n_t)_{t\geq 0}$
converges to a limit $(f_t)_{t\geq 0}$ in $C([0,\infty),\cP(\rd))$; this also implies that  $(f^n_t\otimes f^n_t)_{t\geq 0}$ tends to $(f_t\otimes f_t)_{t\geq 0}$.
Hence for all $T>0$,
all $\psi \in C^2_b(\rd)$ and all $\Psi \in C^2_b(\rd\times\rd)$,
\begin{equation}\label{ttty}
\sup_{[0,T]} \Big[ \Big|\intrd \psi(v) (f^n_t(\dd v)-f_t(\dd v))\Big|+
\Big|\intrdrd \Psi(v,\vs) (f^n_t(\dd v)f^n_t(\dd \vs)-
f_t(\dd v)f_t(\dd \vs))\Big| \to 0 
\end{equation}
as $n\to \infty$. It remains to check that this limit is indeed a weak solution to \eqref{LE} 
starting from $f_0$.

\vip
First, using the uniform integrability property \eqref{eq: UI in existence proof},
$$
\sup_{n\geq n_0}\sup_{t\in [0,T]} \intrd h(|v|^2)f^n_t (\dd v) <\infty
$$
and recalling that $\lim_{r\to \infty} r^{-1}h(r) = \infty$, one easily check that 
for all $t\geq 0$, $m_2(f_t)=\lim_n m_2(f^n_t)$. Since now $m_2(f^n_t)=m_2(f^n_0)\to m_2(f_0)$,
we deduce that $(f_t)_{t\geq 0}$ is energy-conserving as desired.

\vip

Next, we fix $\varphi \in C^2_b(\rd)$ and recall that $\cL \varphi$ is continuous on
$\rd\times\rd$ and satisfies the growth bound $|\cL\varphi(v,\vs)| \leq C(1+|v|^{2+\gamma}+|\vs|^{2+\gamma})$.
We can then let $n\to \infty$ in the formula 
$$
\intrd \varphi(v)f_t^n(\dd v) = \intrd \varphi(v)f_0^n(\dd v) + \intot \intrd \intrd \cL\varphi(v,\vs) 
f_s^n(\dd \vs)f_s^n(\dd v) \dd s,
$$
and conclude that \eqref{wf} is satisfied,
using \eqref{ttty} and the uniform integrability given by \eqref{eq: UI in existence proof} 
(recall that $\lim_{r\to \infty} h'(r)=\infty$), i.e.
$$ 
\sup_{n\geq n_0}\int_0^t \intrd |v|^{2+\gamma}h'(|v|^2)f^n_s (\dd v) \dd s <\infty.
$$
\bla The proof is complete.
\end{proof}

\section{Regularity}\label{pf of regularity}  
We now prove our regularity result Theorem \ref{mainregularity}. We begin with the following very 
mild regularity principle, which guarantees that the hypotheses of Theorem \ref{ddv}-(c) apply at 
some small time, provided that $f_0$ has $4$ moments. We then `bootstrap' to the claimed result, 
using Theorems \ref{ddv} and \ref{analytic regularity} \bla and our uniqueness result. 

\begin{lem}\label{weak regularity} Let $\gamma \in (0,1]$ and \bla $f_0\in \cP_{4}(\rd)$ 
be a measure which is not a Dirac mass, 
and let $(f_t)_{t\ge 0}$ be the weak solution to \eqref{LE} starting at $f_0$. 
Then, for any $t_0>0$, 
there exists $t_1\in [0,t_0)$ such that $f_{t_1}$ is not concentrated on a line. 
\end{lem} 

\begin{proof} If $f_0$ is already not concentrated on a line, there is nothing to prove.
We thus assume that $f_0$ concentrates on a line and, by translational and 
rotational invariance, that $f_0$ concentrates on the $z$-axis
$L_0=\{(0,0,z) : z\in \rr\}$. \bla Further, since $f_0$ is not a point mass, we can find two disjoint compact 
intervals $K_1,K_2 \subset L_0$ such that $f_0(K_1)>0$ and  $f_0(K_2)>0$.

\vip
{\bf Step 1.} 
We introduce the following averaged coefficients: for $v\in \rd$ and $f\in \cP_2(\rd)$, define 
$$ b(v, f)=\intrd b(v-\vs)f(\dd \vs), \qquad a(v,f)=\intrd a(v-\vs)f(\dd \vs)$$ 
and let $\sigma(v,f)$ be a square root of $a(v,f)$. Now, let $(B_t)_{t\geq 0}$ be a 3-dimensional 
Brownian motion, and 
$V_0$ an independent random variable in $\rd$. From \cite[Proposition 10]{fgui}, the It\^o stochastic 
differential equation 
\begin{equation}\label{eq: SDE} 
V_t=V_0+\intot b(V_s, f_s)\dd s+\intot \sigma(V_s, f_s)\dd B_s
\end{equation} 
has a pathwise unique solution and, if $V_0$ is $f_0$-distributed, then 
$V_t\sim f_t$ for all $t\ge 0$. We will denote by $\PR_{v_0}$, $\E_{v_0}$ the probability and
expectation concerning the process started from the 
deterministic initial condition $V_0=v_0$. We thus have
$f_t(A)=\intrd \PR_{v_0}(V_t \in A) f_0(\dd v)$ for any $A \in {\mathcal B}(\rd)$,
any $t\geq 0$. \bla

\vip
{\bf Step 2.}
We now claim that if $F:\rd \rightarrow \mathbb{R}$ is bounded and continuous 
and $Z\sim \mathcal{N}(0, I_3)$, then
\begin{equation} \label{eq: uniform convergence in law} 
\lim_{\e\to 0}\sup_{v_0\in K_1}\Big|\E_{v_0}\Big[F\Big(\frac{V_\e-v_0}{\sqrt{\e}}\Big)\Big]
-\E \Big[F\Big(\sigma(v_0, f_0)Z\Big)\Big]\Big|=0.
\end{equation} 
Let $U$ be an open ball containing $K_1$, and for $v\in \rd$, let $\pi(v)$ be the unique 
minimiser of $|v-\tv|$ over $\tv\in \overline{U}$. 
Recalling the growth bounds 
$$|b(v-\vs)|\le C|v-\vs|^{1+\gamma}, 
\qquad \|a(v-\vs)\|\le C|v-\vs|^{2+\gamma},
$$ 
that $\sup_{t\geq 0} m_4(f_t)<\infty$ by Theorem \ref{ddv}-(a), 
one checks that $|b(v,f_s)|+||\sigma(f_s,v)|| \leq C(1+|v|^{1+\gamma})$ and, since
$f_t\to f_0$ weakly as $t\to 0$, that $a(v, f_t)\to a(v, f_0)$, and thus  
$\sigma(v,f_t)\to \sigma(v,f_0)$, uniformly over $v\in \overline{U}$, as $t\to 0$.
We now define 
$$ b_t(v)=b(\pi(v), f_t); \qquad \sigma_t(v)=\sigma(\pi(v), f_t)$$ 
so that $b_t(v)$ and $\sigma_t(v)$ are bounded, globally Lipschitz in $v$, 
agree with $b(v, f_t), \sigma(v,f_t)$ for 
$v \in \overline{U}$ and $\sigma_t(v)$ converges uniformly on $\rd$ as $t\downarrow 0$.
Now, let $\tV_t$ be the solution to 
the stochastic differential equation \eqref{eq: SDE} with these coefficients in place of 
$b(v, f_t)$ and $\sigma(v,f_t)$, and let $T$ be the stopping time when $\tV_t$ first leaves $U$. 
By uniqueness, we have $V_t=\tV_t$ for all $t\in [0,T]$. Using now that
$b_t$ and $\sigma_t$ are bounded, that $\sigma_t \to \sigma_0$ uniformly
and that $\tV_t \to v_0$ as $t\to 0$, \bla we see that
\begin{align}\label{klm}
&\limsup_{\e \to 0} \sup_{v_0 \in K_1} \E_{v_0}\Big[ \Big| \frac{\tV_\e - v_0}{\sqrt \e} - 
\sigma_0(v_0) \frac{B_\e}{\sqrt \e}\Big|^2 \Big]\\
\leq& \limsup_{\e \to 0} \sup_{v_0 \in K_1} \frac 1\e \E_{v_0}\Big[2\Big(\int_0^\e b_s(\tV_s)\dd s\Big)^2
+2\Big(\int_0^\e (\sigma_s(\tV_s)-\sigma_0(v_0)) \dd B_s\Big)^2\Big]=0. \notag
\end{align}
Recalling that $\sigma_0(v_0)=\sigma(v_0,f_0)$ when $v_0 \in K_1$ \bla 
and that $\frac{B_\e}{\sqrt \e}\sim \mathcal{N}(0, I_3)$,
we conclude that
\begin{align*} &\sup_{v_0\in K_1}\Big|\E_{v_0}\Big[F\Big(\frac{V_\e-v_0}{\sqrt{\e}}\Big)\Big]
-\E_{v_0}[F(\sigma(v_0,f_0) Z)]\Big|\\
\le&  \sup_{v_0\in K_1}\Big|\E_{v_0}\Big[F\Big(\frac{\tV_\e-v_0}{\sqrt{\e}}\Big)\Big]
-\E_{v_0}\Big[F\Big(\sigma_0 (v_0)\frac{B_\e}{\sqrt \e}\Big)\Big]\Big|
+ 2\|F\|_\infty \hspace{0.1cm} \sup_{v_0\in K_1} \PR(T<\e)\rightarrow 0 
\end{align*} 
where the final convergence follows \eqref{klm} and the fact that
$\sup_{v_0\in K_1} \PR(T<\e)\to 0$ because
$d(K_1, U^\mathrm{c})=\inf\{|v-\tv|:v\in K_1, \tv\not \in U\}>0$ and because
$b_t$ and $\sigma_t$ are bounded.
The proof of the claim is complete. 

\vip 

{\bf Step 3.} We now construct three test functions $F_i$ to which apply Step 2:
let $B_i\subset \mathbb{R}^2, i=1,2,3$ be disjoint open balls
in the plane such that no line (in the plane) meets all three, and let 
$\chi_i:\mathbb{R}^2\rightarrow [0,1]$ be nonzero, smooth 
bump functions, supported on each $B_i$. Now, we define $\rho:\rd\to\rr^2$
the projection $\rho(v_1,v_2,v_3)=(v_1,v_2)$. We then introduce the
bounded smooth functions
$F_i:\rd\rightarrow [0,1]$ defined \bla by $F_i(v)=\chi_i(\rho(v))$. Observe that
$F_i(v)\leq \indiq_{\{\rho(v)\in B_i\}}$.

\vip

Since $f_0$ concentrates on the $z$-axis $L_0$, denoting by
$e_3=(0,0,1)$, we have, for all $v_0 \in L_0$,
$$
a(v_0,f_0)=\intrd |v_0-v|^{\gamma+2}\Pi_{(v-v_0)^\perp}f_0(\dd v) =
h(v_0) \Pi_{e_3^\perp} = h(v_0)\begin{pmatrix} 1 & 0 & 0 \\ 0 
& 1 & 0 \\ 0 &0&0
\end{pmatrix} ,
$$
where $h(v_0)=\intrd |v_0-v|^{\gamma+2}f_0(\dd v)$. One easily checks that $h$
is bounded from above and from below on $K_1$, since
$\sup_{v_0 \in K_1} h(v_0) \leq C(1+m_{2+\gamma}(f_0))$
and $\inf_{v_0 \in K_1} h(v_0) \geq \alpha^{\gamma+2}f_0(K_2)$, where $\alpha>0$
is the distance between $K_1$ and $K_2$. Since $\sigma(v_0,f_0)=[a(v_0,f_0)]^{1/2}$
and since $\rho(Z)\sim \mathcal{N}(0, I_2)$, we deduce that
for some $\delta>0$ and all $i=1,2,3$,
$$ \inf_{v_0\in K_1} \E_{v_0} [F_i(\sigma(v_0, f_0)Z)]= \inf_{v_0\in K_1} \E [\chi_i(h^{1/2}(v_0)\rho(Z))]
\ge 2\delta>0.$$ 
Thanks 
to \eqref{eq: uniform convergence in law}, we can find $\e_0>0$ such that for all
$\e\in (0,\e_0)$, all $i=1,2,3$,
$$ \inf_{v_0\in K_1} \E_{v_0}\Big[F_i\Big(\frac{V_{\e}-v_0}{\sqrt{\e}}\Big)\Big]\ge \delta
\quad \hbox{whence}\quad 
\inf_{v_0\in K_1}\PR_{v_0}\Big(\rho\Big(\frac{V_{\e}-v_0}{\sqrt{\e}}\Big)\in B_i \Big)\ge \delta.
$$

{\bf Step 4.} Now, we fix $t_0>0$ as in the statement, and consider $t_1 \in (0,\e_0\land t_0)$.
For a given line $L=\{x_0+\lambda u_0 : \lambda \in \rr\}\subset\rd$ and
for $v_0\in K_1$, we denote by
$L_{t_1,v_0}=\rho((L-v_0)/\sqrt{t_1})$, which is a line
(or a point) in $\rr^2$. There is $i\in \{1,2,3\}$, possibly depending on $t_1$
and on $v_0$, such that $L_{t_1,v_0} \cap B_i=\emptyset$, so that
\begin{align*}
\PR_{v_0}(V_{t_1} \in L)= &\PR_{v_0}\Big( \frac{V_{t_1}-v_0}{\sqrt{{t_1}}} 
\in \frac{L-v_0}{\sqrt{{t_1}}}\Big)\\
\leq & \PR_{v_0}\Big( \rho\Big(\frac{V_{t_1}-v_0}{\sqrt{{t_1}}}\Big) \in L_{{t_1},v_0}\Big)\\
\leq& 1 - \PR_{v_0}\Big( \rho\Big(\frac{V_{t_1}-v_0}{\sqrt{{t_1}}}\Big) \in B_i\Big)\\
\leq &1-\delta
\end{align*}
by Step 3. \bla
In other words, for all $v_0 \in K_1$, $\PR_{v_0}(V_{t_1} \in \rd\setminus L) 
\geq \delta$, whence
$$ f_{t_1}(\rd\setminus L)=\intrd \PR_{v_0}(V_{t_1}\not \in L)f_0(\dd v_0) \ge \delta f_0(K_1)>0.$$ 
The proof is complete.
\end{proof}  

We now prove our claimed result.

\begin{proof}[Proof of Theorem \ref{mainregularity}] Let 
$f_0\in \cP_2(\rd)$ \bla not be a point mass, 
and let $(f_t)_{t\ge 0}$ be any weak solution to \eqref{LE} starting at $f_0$. Fix $t_0>0$. 
By Theorem \ref{ddv}-(a), picking $t_1\in (0, t_0)$ arbitrarily, we have $m_{4}(f_{t_1})<\infty$ 
and, 
due to conservation of energy and momentum, $f_{t_1}$ is not a point mass. We can therefore 
apply Lemma 
\ref{weak regularity} to find $t_2\in [t_1, t_0)$ such that $f_{t_2}$ is not concentrated on a line,
and we also have $m_{4}(f_{t_2})<\infty$, still by Theorem \ref{ddv}-(a), because $t_2>0$.
\vip

Now, by Theorem \ref{ddv}-(c), there exists \emph{a} solution $(g_t)_{t\ge 0}$ to \eqref{LE} starting at 
$g_0=f_{t_2}$ such that, for all $s, k\ge 0$ and $\delta>0$,
$$\sup_{t\ge \delta}\|g_t\|_{H^k_s(\rd)}<\infty$$
and such that $H(g_t)<\infty$ for all $t>0$; by Theorem \ref{analytic regularity}, $g_t$ is further analytic 
for all $t>0$. 
\vip

By uniqueness, see Theorem \ref{main} and recall that $m_{4}(f_{t_2})<\infty$,
there is a unique weak solution to \eqref{LE} starting at $g_0=f_{t_2}$, whence $g_t=f_{t_2+t}$ for 
all $t\ge 0$. In particular, $f_{t_0}=g_{t_0-t_2}$ is analytic and has finite entropy and (choosing $\delta=t_0-t_2$),
for all  $s, k\ge 0$, $\sup_{t\ge t_0}\|f_t\|_{H^k_s(\rd)}<\infty$.
\end{proof} 

\section{Proof of the central inequality} \label{proof of cent}

We finally handle the \bla

\begin{proof}[Proof of Lemma \ref{cent}]
We introduce the shortened notation $x=v-\vs$, $\tx=\tv-\tvs$ and recall that
\begin{align}\label{i0}
\cLL c_{p,\e}(v,\vs,\tv,\tvs)
 \leq  k_{p,\e}^{(1)}+k_{p,\e}^{(2)}+ \tk_{p,\e}^{(2)}+k_{p,\e}^{(3)}+\tk_{p,\e}^{(3)},
\end{align}
where $k_{p,\e}^{(1)}=k_{p,\e}^{(1)}(v,\vs,\tv,\tvs)$, $k_{p,\e}^{(2)}=k_{p,\e}^{(2)}(v,\vs,\tv,\tvs)$,
$\tk_{p,\e}^{(2)}=k_{p,\e}^{(2)}(\tv,\tvs,v,\vs)$, etc.
In the whole proof, $C$ is allowed to change from line to line and to depend (only) 
on $p$ and $\gamma$.

\vip
{\bf Step 1.} Here we show, and this is the most tedious estimate, that
\begin{align}\label{i1}
k_{p,\e}^{(1)}\leq& 2c_{p+\gamma,\e}(v,\tv)\\
&+C\sqrt\e (1+|\vs|^p+|\tvs|^p)c_{p+\gamma,\e}(v,\tv) \notag\\
&+ C\sqrt\e (1+|v|^p+|\tv|^p) c_{p+\gamma,\e}(\vs,\tvs)\notag\\
&+ \frac{C}{\sqrt \e} (1+|\vs|^{p+\gamma}+|\tvs|^{p+\gamma})c_{p,\e}(v,\tv) \notag\\
&+ \frac{C}{\sqrt \e} (1+|v|^{p+\gamma}+|\tv|^{p+\gamma})c_{p,\e}(\vs,\tvs).\notag
\end{align}

We start from
$$
k_{p,\e}^{(1)}=(1+|v|^p+|\tv|^p)\varphi_\e'(|v-\tv|^2)[g_1+g_2+g_3],
$$
where
\begin{align*}
g_1=&[(v-\tv)-(\vs-\tvs)]\cdot(b(x)-b(\tx))  + ||\sigma(x)-\sigma(\tx)||^2,\\
g_2=&(v-\tv)\cdot(b(x)-b(\tx)) ,\\
g_3=&(\vs-\tvs)\cdot (b(x)-b(\tx)).
\end{align*}

{\it Step 1.1.} Recalling that $b(x)=-2|x|^\gamma x$ and using \eqref{p3}, we find
\begin{align*}
g_1 \leq& 2(x-\tx)\cdot[-|x|^\gamma x+|\tx|^\gamma\tx] 
+ 2|x|^{\gamma+2}+2|\tx|^{\gamma+2}-4|x|^{\gamma/2}|\tx|^{\gamma/2}(x\cdot \tx) \\
=& 2(|x|^\gamma+|\tx|^\gamma ) (x\cdot\tx) -4 |x|^{\gamma/2}|\tx|^{\gamma/2}(x\cdot \tx) \\
=& 2 (x\cdot\tx) (|x|^{\gamma/2}-|\tx|^{\gamma/2})^2.
\end{align*}
Using now \eqref{ttaacc} with $\alpha=\gamma/2$,
$$
g_1 \leq 2 |x||\tx|(|x|\lor|\tx|)^{\gamma-2}(|x|-|\tx|)^2=
2 (|x|\land|\tx|)(|x|\lor|\tx|)^{\gamma-1}(|x|-|\tx|)^2
 \leq 2 (|x|\land|\tx|)^\gamma |x-\tx|^2.
$$
Since $|x-\tx|=|(v-\tv)-(\vs-\tvs)|$, we end with
$$
g_1 \leq 2(|x|\land|\tx|)^\gamma |v-\tv|^2+ 2(|x|\land|\tx|)^\gamma (2|v-\tv||\vs-\tvs| + |\vs-\tvs|^2).
$$

{\it Step 1.2.} We next study $g_2$, assuming without loss of generality that $|x|\geq |\tx|$. We write, using \eqref{ttaacc}
with $\alpha=\gamma$,
\begin{align*}
g_2=&2(v-\tv)\cdot[-|x|^\gamma (x-\tx)+(|\tx|^\gamma-|x|^\gamma)\tx]\\
\leq& -2|x|^\gamma(v-\tv)\cdot (x-\tx) + 2 |v-\tv| |\tx| (|x|\lor|\tx|)^{\gamma-1}||x|-|\tx||\\
\leq & -2|x|^\gamma(v-\tv)\cdot (x-\tx) + 2 |v-\tv| |\tx|^{\gamma}|x-\tx|.
\end{align*}
Since now $x=v-\vs$ and $\tx=\tv-\tvs$, we see that
\begin{align*}
g_2\leq& - 2|x|^\gamma |v-\tv|^2+2|x|^\gamma |v-\tv||\vs-\tvs|+ 2|\tx|^\gamma[|v-\tv|^2+|v-\tv||\vs-\tvs|]\\
\leq& 2(|x|^\gamma+|\tx|^\gamma) |v-\tv||\vs-\tvs|
\end{align*}
since $|x|\geq |\tx|$ by assumption. By symmetry, the same bound
holds when $|x|\leq |\tx|$.
\vip

{\it Step 1.3.} Using now \eqref{p2}, we see that
\begin{align*}
g_3\leq& 2|\vs-\tvs|[|x|^\gamma+|\tx|^\gamma] |x-\tx|
\leq 2(|x|^\gamma+|\tx|^\gamma) [|v-\tv||\vs-\tvs|+|\vs-\tvs|^2].
\end{align*}

{\it Step 1.4.} Gathering Steps 1.1, 1.2, 1.3, we have checked that 
$$
k_{p,\e}^{(1)} \leq (1+|v|^p+|\tv|^p)\varphi_\e'(|v-\tv|^2)
\Big[2(|x|\land|\tx|)^\gamma |v-\tv|^2 +C(|x|^\gamma+|\tx|^\gamma) (|v-\tv||\vs-\tvs|+|\vs-\tvs|^2)\Big].
$$
Recalling that $r\varphi_\e'(r)\leq \varphi_\e(r)$ by \eqref{ve} and that $|x|^\gamma\leq|v|^\gamma+|\vs|^\gamma$
and $|\tx|^\gamma\leq|\tv|^\gamma+|\tvs|^\gamma$, we may write $k_{p,\e}^{(1)}\leq k_{p,\e}^{(11)}+ k_{p,\e}^{(12)}$,
where
\begin{align*}
k_{p,\e}^{(11)}=& 2(1+|v|^p+|\tv|^p)[(|v|^\gamma+|\vs|^\gamma)\land (|\tv|^\gamma+|\tvs|^\gamma)] \varphi_\e(|v-\tv|^2),\\
k_{p,\e}^{(12)}=& C(1+|v|^p+|\tv|^p)(|v|^\gamma+|\vs|^\gamma+|\tv|^\gamma+|\tvs|^\gamma)\varphi_\e'(|v-\tv|^2)
(|v-\tv||\vs-\tvs|+|\vs-\tvs|^2).
\end{align*}
First, 
\begin{align*}
k_{p,\e}^{(11)}\leq & 2(|v|^\gamma+|\vs|^\gamma)\varphi_\e(|v-\tv|^2) + 2 |v|^p(|v|^\gamma+|\vs|^\gamma)\varphi_\e(|v-\tv|^2)
+ 2 |\tv|^p(|\tv|^\gamma+|\tvs|^\gamma) \varphi_\e(|v-\tv|^2)\\
=& 2(|v|^{p+\gamma}+|\tv|^{p+\gamma}) \varphi_\e(|v-\tv|^2) + 
2(|v|^\gamma+|\vs|^\gamma+|v|^p|\vs|^\gamma+|\tv|^p|\tvs|^\gamma)\varphi_\e(|v-\tv|^2)\\
\leq &  2(|v|^{p+\gamma}+|\tv|^{p+\gamma}) \varphi_\e(|v-\tv|^2) + 
C(1+|\vs|^\gamma+|\tvs|^\gamma)(1+|v|^p+|\tv|^p)\varphi_\e(|v-\tv|^2)\\
=&2 c_{p+\gamma,\e}(v,\tv) + C(1+|\vs|^{p+\gamma}+|\tvs|^{p+\gamma})c_{p,\e}(v,\tv).
\end{align*}
We next use that $ab\leq \e^{1/2} a^2 + \e^{-1/2} b^2$ to write
\begin{align*}
k_{p,\e}^{(12)}\leq & C \sqrt \e (1+|v|^p+|\tv|^p)(|v|^\gamma+|\vs|^\gamma+|\tv|^\gamma+|\tvs|^\gamma)
\varphi_\e'(|v-\tv|^2)|v-\tv|^2 \\
&+ \frac C {\sqrt\e} (1+|v|^p+|\tv|^p)(|v|^\gamma+|\vs|^\gamma+|\tv|^\gamma+|\tvs|^\gamma)\varphi_\e'(|v-\tv|^2)
|\vs-\tvs|^2 \\
\leq& C \sqrt \e (1+|v|^p+|\tv|^p)(|v|^\gamma+|\vs|^\gamma+|\tv|^\gamma+|\tvs|^\gamma)
\varphi_\e(|v-\tv|^2) \\
&+ \frac C {\sqrt\e} (1+|v|^p+|\tv|^p)(|v|^\gamma+|\vs|^\gamma+|\tv|^\gamma+|\tvs|^\gamma)|\vs-\tvs|^2,
\end{align*}
because $r\varphi_\e'(r)\leq \varphi_\e(r)$ and $\varphi_\e'(r)\leq 1$ by \eqref{ve}. We carry on with
\begin{align*}
k_{p,\e}^{(12)}\leq & C \sqrt \e (1+|v|^{p+\gamma}+|\tv|^{p+\gamma})\varphi_\e(|v-\tv|^2) 
+ C \sqrt \e (1+|v|^{p}+|\tv|^{p})(|\vs|^\gamma+|\tvs|^\gamma) \varphi_\e(|v-\tv|^2)\\
& + \frac C {\sqrt\e} (1+|v|^p+|\tv|^p)(|v|^\gamma+|\vs|^\gamma+|\tv|^\gamma+|\tvs|^\gamma)
(1+\e|\vs-\tvs|^2)\varphi_\e(|\vs-\tvs|^2)\\
\leq & C \sqrt \e c_{p+\gamma,\e}(v,\tv)
+ C \sqrt \e (1+|\vs|^\gamma+|\tvs|^\gamma) c_{p,\e}(v,\tv)\\
& + \frac C {\sqrt\e} (1+|v|^p+|\tv|^p)(|v|^\gamma+|\vs|^\gamma+|\tv|^\gamma+|\tvs|^\gamma)\varphi_\e(|\vs-\tvs|^2)\\
&+ C \sqrt\e (1+|v|^p+|\tv|^p)(|v|^\gamma+|\vs|^\gamma+|\tv|^\gamma+|\tvs|^\gamma)(|\vs|^2+|\tvs|^{2})
\varphi_\e(|\vs-\tvs|^2)\\
\leq & C \sqrt \e c_{p+\gamma,\e}(v,\tv)
+ C \sqrt \e (1+|\vs|^{p+\gamma}+|\tvs|^{p+\gamma}) c_{p,\e}(v,\tv)\\
& + \frac C {\sqrt\e} (1+|v|^{p+\gamma}+|\tv|^{p+\gamma})(1+|\vs|^\gamma+|\tvs|^\gamma)\varphi_\e(|\vs-\tvs|^2)\\
&+ C \sqrt\e (1+|v|^{p+\gamma}+|\tv|^{p+\gamma})(|\vs|^2+|\tvs|^{2}) \varphi_\e(|\vs-\tvs|^2)\\
&+ C \sqrt\e (1+|v|^{p}+|\tv|^{p})(1+|\vs|^{2+\gamma}+|\tvs|^{2+\gamma})\varphi_\e(|\vs-\tvs|^2).
\end{align*}
Since $p\geq 2$, since $\gamma\in (0,1)$ and since $\e\in(0,1]$, we end with
\begin{align*}
k_{p,\e}^{(12)}\leq & C \sqrt \e c_{p+\gamma,\e}(v,\tv)
+C(1+|\vs|^{p+\gamma}+|\tvs|^{p+\gamma})c_{p,\e}(v,\tv)\\
&+ \frac C {\sqrt\e} (1+|v|^{p+\gamma}+|\tv|^{p+\gamma})c_{p,\e}(\vs,\tvs)\\
&+ C \sqrt \e (1+|v|^{p}+|\tv|^{p})c_{p+\gamma,\e}(\vs,\tvs).
\end{align*}
Summing the bounds on $k_{p,\e}^{(11)}$ and $k_{p,\e}^{(12)}$ leads us to \eqref{i1}.

\vip

{\bf Step 2.}
We next prove that
\begin{align}\label{i2p}
k^{(2)}_{p,\e}\leq - p |v|^{p+\gamma}\varphi_{\e}(|v-\tv|^2)+ C(1+|\vs|^{p+\gamma})c_{p,\e}(v,\tv),
\end{align}
and this will imply, still allowing $C$ to change from line to line and to depend on $p$, that
\begin{align}\label{i2}
k^{(2)}_{p,\e}+\tk^{(2)}_{p,\e} \leq& - p (|v|^{p+\gamma}+|\tv|^{p+\gamma})\varphi_{\e}(|v-\tv|^2)
+ C(1+|\vs|^{p+\gamma}+|\tvs|^{p+\gamma})c_{p,\e}(v,\tv)\\
= & - p (c_{p+\gamma}(v,\tv)-\varphi_\e(|v-\tv|^2))
+ C(1+|\vs|^{p+\gamma}+|\tvs|^{p+\gamma})c_{p,\e}(v,\tv)\notag\\\le  & - p c_{p+\gamma}(v,\tv)
+ C(1+|\vs|^{p+\gamma}+|\tvs|^{p+\gamma})c_{p,\e}(v,\tv),\notag
\end{align}
where the equality uses the definition \eqref{cpe} of $c_{p+\gamma, \e}$, and in
the final line we absorb $\varphi(|v-\tv|^2)\le c_{p,\e}(v,\tv)$ into the second term.
\vip

By \eqref{tto}-\eqref{tto2} and by definition of $k^{(2)}_{p,\e}$, we see that
\begin{align*}
k^{(2)}_{p,\e} \leq& \varphi_{\e}(|v-\tv|^2)\Big[ -p |v|^{p+\gamma}  +p|v|^p|\vs|^\gamma
 + C p^2 (|v|^{p-2+\gamma}|\vs|^2+|v|^{p-2}|\vs|^{2+\gamma})\Big] .\\
\leq & \varphi_{\e}(|v-\tv|^2)\Big[ -p |v|^{p+\gamma}  + C(1+|\vs|^{2+\gamma})(1+|v|^p)\Big],
\end{align*}
from which \eqref{i2p} follows.
\vip

{\bf Step 3.} 
We finally prove that
\begin{align}\label{i3}
k^{(3)}_{p,\e}+\tk^{(3)}_{p,\e} \leq & 
C (1+|\vs|^{p+\gamma}+|\tvs|^{p+\gamma})c_{p,\e}(v,\tv) \\
&+C(1+|v|^{p+\gamma}+|\tv|^{p+\gamma})c_{p,\e}(\vs,\tvs) \notag\\
&+ C\sqrt\e(1+|\vs|^{p}+|\tvs|^{p})c_{p+\gamma,\e}(v,\tv)\notag\\
&+ C\sqrt\e(1+|v|^{p}+|\tv|^{p})c_{p+\gamma,\e}(\vs,\tvs).\notag
\end{align}
By symmetry, it suffices to treat the case of $k^{(3)}_{p,\e}$.
Recalling that $|\sigma(x)v|\leq C |x|^{\gamma/2}|v||\vs|$ by \eqref{tr}, and that 
$||\sigma(x)-\sigma(\tx)||\leq C(|x|^{\gamma/2}+|\tx|^{\gamma/2})|x-\tx|$ by \eqref{p4},
we directly find
\begin{align*}
k^{(3)}_{p,\e}\leq& C |v|^{p-1}|\vs||x|^{\gamma/2}(|x|^{\gamma/2}+|\tx|^{\gamma/2})|x-\tx||v-\tv|\varphi_\e'(|v-\tv|^2) \\
\leq & C |v|^{p-1}|\vs|(|v|^\gamma+|\tv|^\gamma+|\vs|^\gamma+|\tvs|^\gamma)
(|v-\tv|^2+|v-\tv||\vs-\tvs|)\varphi_\e'(|v-\tv|^2)\\
=& k^{(31)}_{p,\e}+k^{(32)}_{p,\e},
\end{align*}
where
\begin{align*}
k^{(31)}_{p,\e} =& C |v|^{p-1}|\vs|(|v|^\gamma+|\tv|^\gamma+|\vs|^\gamma+|\tvs|^\gamma)|v-\tv|^2\varphi_\e'(|v-\tv|^2),\\
k^{(32)}_{p,\e} =& C |v|^{p-1}|\vs|(|v|^\gamma+|\tv|^\gamma+|\vs|^\gamma+|\tvs|^\gamma)|v-\tv||\vs-\tvs|
\varphi_\e'(|v-\tv|^2).
\end{align*}
Since $r\varphi_\e'(r)\leq \varphi_\e(r)$ by \eqref{ve}, we have
\begin{align*}
k^{(31)}_{p,\e} \leq& C(1+|\vs|^{1+\gamma}+|\tvs|^{1+\gamma})(1+|v|^{p-1+\gamma}+|\tv|^{p-1+\gamma})\varphi_\e(|v-\tv|^2)\\
\leq& C(1+|\vs|^{p+\gamma}+|\tvs|^{p+\gamma})c_{p,\e}(v,\tv).
\end{align*}
Next, we use that, with $a=|v-\tv|$ and $a_*=|\vs-\tvs|$, since
$a[\varphi_\e'(a^2)]\leq \sqrt{a^2 \varphi_\e'(a^2)} \leq
\sqrt{\varphi_\e(a^2)}$ by \eqref{ve},
$$
a a_* \varphi_\e'(a^2) \leq \sqrt{\varphi_\e(a^2)} \sqrt{\varphi_\e(a_*^2)(1+\e a_*^2)}
\leq [\varphi_\e(a^2)+\varphi_\e(a_*^2)](1+\sqrt{\e}a_*)
$$
to write $k^{(32)}_{p,\e}\leq k^{(321)}_{p,\e}+k^{(322)}_{p,\e}+k^{(323)}_{p,\e}+k^{(324)}_{p,\e}$, where
\begin{align*}
k^{(321)}_{p,\e} =& C |v|^{p-1}|\vs|(|v|^\gamma+|\tv|^\gamma+|\vs|^\gamma+|\tvs|^\gamma)\varphi_\e(|v-\tv|^2),\\
k^{(322)}_{p,\e} =& C |v|^{p-1}|\vs|(|v|^\gamma+|\tv|^\gamma+|\vs|^\gamma+|\tvs|^\gamma)\varphi_\e(|\vs-\tvs|^2),\\
k^{(323)}_{p,\e} =& C \sqrt\e |v|^{p-1}|\vs|(|v|^\gamma+|\tv|^\gamma+|\vs|^\gamma+|\tvs|^\gamma)
|\vs-\tvs|\varphi_\e(|v-\tv|^2),\\
k^{(324)}_{p,\e} =& C \sqrt\e |v|^{p-1}|\vs|(|v|^\gamma+|\tv|^\gamma+|\vs|^\gamma+|\tvs|^\gamma)
|\vs-\tvs|\varphi_\e(|\vs-\tvs|^2).\\
\end{align*}
We have
\begin{align*}
k^{(321)}_{p,\e} \leq& C(1+|\vs|^{1+\gamma}+|\tvs|^{1+\gamma})(1+|v|^{p-1+\gamma}+|\tv|^{p-1+\gamma}) \varphi_\e(|v-\tv|^2)\\
\leq& C(1+|\vs|^{p+\gamma}+|\tvs|^{p+\gamma}) c_{p,\e}(v,\tv),
\end{align*}
as well as 
\begin{align*}
k^{(322)}_{p,\e} \leq& C(1+|\vs|^{1+\gamma}+|\tvs|^{1+\gamma})(1+|v|^{p-1+\gamma}+|\tv|^{p-1+\gamma}) \varphi_\e(|\vs-\tvs|^2)\\
\leq& C(1+|v|^{p+\gamma}+|\tv|^{p+\gamma}) c_{p,\e}(\vs,\tvs),
\end{align*}
and,  dropping $\sqrt \e$ and using that $|\vs-\tvs|\leq |\vs|+|\tvs|$,
\begin{align*}
k^{(323)}_{p,\e} \leq& C(1+|\vs|^{2+\gamma}+|\tvs|^{2+\gamma})(1+|v|^{p-1+\gamma}+|\tv|^{p-1+\gamma}) \varphi_\e(|v-\tv|^2)\\
\leq& C(1+|\vs|^{p+\gamma}+|\tvs|^{p+\gamma}) c_{p,\e}(v,\tv).
\end{align*}
Finally,  using again the bound $|\vs-\tvs|\leq|\vs|+|\tvs|$, 
\begin{align*}
k^{(324)}_{p,\e} \leq& C\sqrt\e (1+|v|^{p-1+\gamma}+|\tv|^{p-1+\gamma})
(1+|\vs|^{2+\gamma}+|\tvs|^{2+\gamma})\varphi_\e(|\vs-\tvs|^2)\\
\leq& C\sqrt\e (1+|v|^{p}+|\tv|^{p})c_{p+\gamma,\e}(\vs,\tvs).
\end{align*}
 Summing the bounds on $k^{(31)}_{p,\e}$, $k^{(321)}_{p,\e}$, $k^{(322)}_{p,\e}$, $k^{(323)}_{p,\e}$
and $k^{(324)}_{p,\e}$ ends the step. 

\vip

Gathering \eqref{i0}, \eqref{i1}, \eqref{i2} and \eqref{i3} completes the proof since $\e\in(0,1]$.
\end{proof}


\begin{thebibliography}{99}


\bibitem{abl}{\sc Alonso, R., Bagland, V. and Lods, B.}
\newblock Long time dynamics for the Landau-Fermi-Dirac equation with hard potentials.
\newblock {\em J. Differential Equations}, 270 (2021), 596--663.
\bla

\bibitem{agt}
{\sc Alonso, R., Gamba, I.M. and Taskovic, M.}
\newblock Exponentially-tailed regularity and time asymptotic for the homogeneous 
Boltzmann equation.
\newblock {\em arXiv:1711.06596.}

\bibitem{ar} {\sc Arsen'ev, A.A. and Buryak, O.E.}
\newblock On the connection between a solution of the Boltzmann equation and a solution of the Landau-Fokker-Planck equation. 
\newblock {\em Mathematics of the USSR-Sbornik.} 69 (1991), 465.

\bibitem{b}
{\sc Bobylev, A.V.}
\newblock Moment inequalities for the Boltzmann equation and applications to spatially
homogeneous problems.
\newblock {\em J. Statist. Phys.} 88 (1997), 1183--1214.

\bibitem{c}
{\sc Carrapatoso, K.}
\newblock Exponential convergence to equilibrium for the homogeneous Landau 
equation with hard potentials.
\newblock {\em Bull. Sci. Math.} 139 (2015), 777--805.

\bibitem{cddw}
{\sc Carrillo, J.A., Delgadino, M.G., Desvillettes, L. and Wu, J.}
\newblock The Landau equation as a Gradient Flow.
\newblock {\em arXiv:2007.08591.}


\bibitem{ch}
{\sc Chen, H., Li, W. and Xu, C.J.}
\newblock Gevrey regularity for solution of the spatially homogeneous Landau equation.
\newblock {\em  Acta Math. Sci. Ser. B} 29 (2009), 673--686.
\bla

\bibitem{ch2}
{\sc Chen, H., Li, W.X. and Xu, C.J.}
\newblock Analytic smoothness effect of solutions for spatially homogeneous Landau equation.
\newblock {\em J. Differential Equations} 248 (2010), 77--94.


\bibitem{d}
{\sc Desvillettes, L.}
\newblock Entropy dissipation estimates for the Landau equation in the
Coulomb case and applications.
\newblock {\em J. Funct. Anal.} 269 (2015), 1359--1403.

\bibitem{d2}
{\sc Desvillettes, L.}
\newblock On asymptotics of the Boltzmann equation when the collisions become grazing. 
\newblock {\em Transport Theory Statist. Phys.} 21 (1992), 259--276.


\bibitem{dv1}
{\sc Desvillettes and L., Villani, C.}
\newblock On the spatially homogeneous Landau equation for hard potentials, 
Part I : existence, uniqueness and smothness.
\newblock {\em Comm. Partial Differential Equations} 25 (2000), 179--259.

\bibitem{dv2}
{\sc Desvillettes and L., Villani, C.}
\newblock On the spatially homogeneous Landau equation for hard potentials, Part II: H-Theorem and Applications.
\newblock {\em Comm. Partial Differential Equations} 25 (2000), 261--298.

\bibitem{fc}
{\sc Fournier, N.}
\newblock Uniqueness of bounded solutions for the homogeneous Landau equation with a Coulomb
potential.
\newblock {\em Comm. Math. Phys.} 299 (2010), 765--782.

\bibitem{fgue}
{\sc Fournier, N. and Gu\'erin, H.}
\newblock Well-posedness of the spatially homogeneous Landau equation for soft potentials.
\newblock {\em J. Funct. Anal.} 256 (2009), 2542--2560.

\bibitem{fgui}
{\sc Fournier, N. and Guillin, A.}
\newblock From a Kac-like particle system to the Landau equation for hard potentials and Maxwell molecules.
\newblock {\em Ann. Sci. \'Ec. Norm. Sup\'er.} 50 (2017), 157--199.

\bibitem{fh}
{\sc Fournier, N. and Hauray, M.}
\newblock Propagation of chaos for the Landau equation with moderately soft potentials.
\newblock {\em Ann. Probab.} 44 (2016), 3581--3660.

\bibitem{fmi} 
{\sc Fournier, N. and Mischler, S.}
\newblock Rate of convergence of the Nanbu particle system for hard potentials and Maxwell molecules. 
\newblock {\em Ann. Probab.} 44(1) (2016), 589--627.


\bibitem{fm}
{\sc Fournier, N. and Mouhot, C.}
\newblock   On the well-posedness of the spatially homogeneous Boltzmann equation
with a moderate angular singularity.
\newblock {\em Comm. Math. Phys.} 289 (2009), 803--824.


\bibitem{fp}
{\sc Fournier, N. and Perthame, B.}
\newblock Monge-Kantorovich distance for PDEs: the coupling method.
\newblock {\em EMS Surv. Math. Sci.} 7 (2020), 1--31.
\bla


\bibitem{fo}
{\sc Fournier, N.}
\newblock On exponential moments of the homogeneous Boltzmann equation for hard potentials without cutoff.
\newblock {\em arXiv preprint arXiv:2012.02982}
\bla

\bibitem{f}
{\sc Funaki, T}, 
\newblock The diffusion approximation of the spatially homogeneous Boltzmann equation
\newblock {\em Duke Math. J.} 52 (1985), 1--23.

\bibitem{go}
{\sc Golse, F., Imbert, C., Mouhot, C. and Vasseur, A.}
\newblock Harnack inequality for kinetic Fokker-Planck equations with rough 
coefficients and application to the Landau equation. 
\newblock {\em Ann. Sc. Norm. Super. Pisa Cl. Sci.} (5) 19 (2019), 253--295.

\bibitem{gou}
{\sc Goudon, T.} 
\newblock On Boltzmann Equations and Fokker-Planck Asymptotics: Influence of Grazing Collisions.
\newblock {\em J. Statist. Phys.} 89 (1997), 751--776.

\bibitem{gu}
{\sc Gu\'erin, H.}
\newblock Existence and regularity of a weak function-solution for some Landau equations with a 
stochastic approach.
\newblock {\em Stochastic Process. Appl.} 101 (2002), 303--325.


\bibitem{g}
{\sc Gu\'erin, H.}
\newblock Solving Landau equation for some soft potentials through a probabilistic approach.
\newblock {\em Ann. Appl. Probab.} 13 (2003), 515--539.


\bibitem{guo}
{\sc Guo, Y.}
\newblock The Landau equation in a periodic box,
\newblock {\em Comm. Math. Phys.} 231 (2002), 391--434.

\bibitem{hy}
{\sc He, L. and Yang, X.}
\newblock Well-posedness and asymptotics of grazing collisions limit of Boltzmann equation 
with Coulomb interaction.
\newblock {\em SIAM J. Math. Anal.} 46 (2014), 4104--4165.

\bibitem{h1}
{\sc Heydecker, D.}
\newblock Pathwise convergence of the hard spheres Kac process.
\newblock {\em Ann. Appl. Probab.} 29 (2019), 3062--3127. 

\bibitem{h2}
{\sc Heydecker, D.}
\newblock Kac's Process with Hard Potentials and a Moderate Angular Singularity.
\newblock arXiv:2008.12943

\bibitem{k}
{\sc Kac, M.}
\newblock Foundations of kinetic theory.
\newblock {\em Proceedings of the Third Berkeley Symposium on Mathematical Statistics and Probability, 
1954--1955, vol. III}, University of California Press, 171--197.

\bibitem{mm}
{\sc Mischler, S. and Mouhot, C.}
\newblock Kac's Program in Kinetic Theory. 
\newblock {\em  Invent. Math.} 193 (2013),  1--147.

\bibitem{mmw} 
{\sc Mischler, S., Mouhot, C. and Wennberg, B.}
\newblock A new approach to quantitative propagation of chaos for drift, diffusion and jump processes. 
\newblock{\em Probab. Theory Related Fields} 161 (2015), 1--59.

\bibitem{mw}
{\sc Mischler, S and Wennberg, B.}
\newblock On the spatially homogeneous Boltzmann equation. 
\newblock{\em Ann. Inst. H. Poincar\'e Anal. Non Lin\'eaire} 16 (1999), 467--501.


\bibitem{morimoto}{\sc Morimoto, Y., Pravda-Starov, K. and Xu, C.J.}
\newblock  A remark on the ultra-analytic smoothing properties of the spatially homogeneous Landau equation. 
\newblock {\em Kinet. Relat. Models} 6 (2013), 715--727.

\bibitem{mou}
{\sc Mouhot, C.}
\newblock Explicit coercivity estimates for the linearized Boltzmann and Landau operators.
\newblock{\em  Comm. Partial Differential Equations} 31 (2006), 1321--1348.

\bibitem{n}
{\sc Norris, J.}
\newblock A consistency estimate for Kac's model of elastic collisions in a dilute gas.
\newblock {\em Ann. Appl. Probab.} 26 (2016), 1029--1081.

\bibitem{p}
{\sc Povzner, A. Ja.}
\newblock On the Boltzmann equation in the kinetic theory of gases.
\newblock {\em Mat. Sb. (N.S.)} 58 (1962), 65--86.

\bibitem{t}
{\sc Tanaka, H.}, 
\newblock Probabilistic treatment of the Boltzmann equation of Maxwellian molecules. 
\newblock {\em Z. Wahrsch. und Verw. Gebiete} 46 (1978/79), 67--105.

\bibitem{v:max}
{\sc Villani, C.}
\newblock On the spatially homogeneous Landau equation for Maxwellian molecules. 
\newblock {\em Math. Models Methods Appl. Sci.} 8 (1998), 957--983.

\bibitem{v:nc}
{\sc Villani, C.}, 
\newblock On a new class of weak solutions to the spatially homogeneous Boltzmann and Landau equations.
\newblock {\em Arch. Rational Mech. Anal.}  143 (1998), 273--307.

\bibitem{v:h}
{\sc Villani, C.} 
\newblock A review of mathematical topics in collisional kinetic theory. 
\newblock {\em Handbook of mathematical fluid dynamics Vol. I,}  71--305, North-Holland, Amsterdam, 2002.

 
\bibitem{v: ot}
{\sc Villani, C.}
\newblock Topics in Optimal Transportation.
\newblock {\em American Mathematical Society}, No. 58, 2003.
\bla

\bibitem{w}
{\sc Walsh, J.B.}  
\newblock An introduction to  stochastic partial differential equations.
\newblock \'Ecole d'\'et\'e de Probabilit\'es de Saint-Flour XIV, Lect. Notes in Math. {{1180}},
265-437, 1986.
\end{thebibliography}
\end{document}